\documentclass[english]{article}
\usepackage[T1]{fontenc}
\usepackage[utf8]{luainputenc}
\usepackage{geometry}
\usepackage{color}
\usepackage{babel}
\usepackage{url}
\usepackage{amsthm}
\usepackage{amsmath}
\usepackage{amssymb}
\usepackage{esint}
\usepackage[unicode=true,pdfusetitle,
 bookmarks=true,bookmarksnumbered=false,bookmarksopen=false,
 breaklinks=false,pdfborder={0 0 1},backref=false,colorlinks=true]
 {hyperref}

\makeatletter
\numberwithin{equation}{section}
\theoremstyle{plain}
\newtheorem{thm}{\protect\theoremname}[section]
  \theoremstyle{plain}
  \newtheorem{prop}[thm]{\protect\propositionname}
  \theoremstyle{plain}
  \newtheorem{lem}[thm]{\protect\lemmaname}
  \theoremstyle{remark}
  \newtheorem{claim}[thm]{\protect\claimname}
  \theoremstyle{plain}
  \newtheorem{cor}[thm]{\protect\corollaryname}

\usepackage{indentfirst}

\makeatother

  \providecommand{\claimname}{Claim}
  \providecommand{\corollaryname}{Corollary}
  \providecommand{\lemmaname}{Lemma}
  \providecommand{\propositionname}{Proposition}
\providecommand{\theoremname}{Theorem}

\begin{document}

\title{Tightness of the recentered maximum of log-correlated Gaussian fields}

\author{Javier Acosta, University of Minnesota}
\maketitle
\begin{abstract}
We consider a family of centered Gaussian fields on the $d$-dimensional
unit box, whose covariance decreases logarithmically in the distance
between points. We prove tightness of the recentered maximum of the
Gaussian fields and provide exponentially decaying bounds on the right
and left tails. We then apply this result to a version of the two-dimensional
continuous Gaussian free field.
\end{abstract}

\section{Introduction}

\subsection*{Main result}

Let $\left\{ \left(Y_{\epsilon}^{x}:x\in[0,1]^{d}\right)\right\} _{\epsilon>0}$
be a family of centered Gaussian fields indexed by the $d$-dimensional
unit box $[0,1]^{d}$, where $d$ is any positive integer. Suppose
that the family satisfies, for some constant $0<C_{Y}<\infty$ and
all $x,y\in[0,1]^{d}$, $\epsilon>0$, 

\begin{equation}
\left|Cov\left(Y_{\epsilon}^{x},Y_{\epsilon}^{y}\right)+\log\left(\max\{\epsilon,\left\Vert x-y\right\Vert \}\right)\right|\leq C_{Y}\label{logcorrY}
\end{equation}
and 

\begin{equation}
\mathbb{E}\left[\left(Y_{\epsilon}^{x}-Y_{\epsilon}^{y}\right)^{2}\right]\leq C_{Y}\epsilon^{-1}\left\Vert x-y\right\Vert \text{ if }\left\Vert x-y\right\Vert \leq\epsilon,\label{l2distY}
\end{equation}
where $\left\Vert \cdot\right\Vert $ is Euclidean distance. Display
\eqref{logcorrY} implies that the covariance is logarithmic for distant
points and that the variance is nearly constant. The second condition
is imposed so that the field does not vary too much for close points.
Display \eqref{l2distY}, basic relations between the moments of Gaussian
random variables and Kolmogorov's continuity criterion (see \cite[Theorem 1.4.17]{k})
imply that the fields have continuous modifications. 

When $d=2$, an example of a field satisfying \eqref{logcorrY} and
\eqref{l2distY} is the bulk of the mollified continuous Gaussian
free field (MGFF), which will be defined in Section 3.1, and will
be the object of our attention in Section 3.

Set $m_{\epsilon}=m_{\epsilon,d}=\sqrt{2d}\log(1/\epsilon)-\frac{3/2}{\sqrt{2d}}\log\log(1/\epsilon)$.
The main result of this paper is:
\begin{thm}
\label{main} There exist constants $0<c,C<\infty$ (depending on
$C_{Y}$ and $d$) such that, for all $\epsilon>0$ small enough, 

\begin{equation}
\mathbb{P}\left(\left|\max_{x\in[0,1]^{d}}Y_{\epsilon}^{x}-m_{\epsilon}\right|\geq\lambda\right)\leq Ce^{-c\lambda}\label{rtubY}
\end{equation}
for all $\lambda\geq0$.
\end{thm}
\noindent Theorem \ref{main} implies, in particular, that $\left\{ \max_{x\in[0,1]^{d}}Y_{\epsilon}^{x}-m_{\epsilon}:\epsilon>0\right\} $
is tight and that 
\[
\left|\mathbb{E}\left[\max_{x\in[0,1]^{d}}Y_{\epsilon}^{x}\right]-m_{\epsilon}\right|\leq C
\]
 for some constant $C$ depending on $C_{Y}$ and $d$. 

The main idea of the proof of Theorem \ref{main} is to use Slepian's
Lemma (see \cite[Theorem 2.2.1]{rf}) to compare the maximum of the
field $Y_{\epsilon}$ with the maximum of the modified branching random
walk (MBRW), a field introduced by Bramson and Zeitouni in \cite{bz}.
Since Slepian's Lemma only allows comparison of fields with the same
index set, we will add an appropriately chosen independent continuous
field to the MBRW. Adding an independent continuous field to the MBRW
does not change the maximum much, provided the continuous field is
small and smooth enough. These fields are defined in detail in Section
2.1. After defining the fields, we compare the right and left tails
in Sections 2.2 and 2.3, respectively. We then show, in Section 3,
that Theorem \ref{main} implies tightness of the recentered maximum
of the MGFF.

A comment on constants: $c$ will always denote a small positive constant
and $C$ will always denote a large positive constant. Both constants
are allowed to change from line to line. The dependence of the constants
will be explicit or will be clear from the context. The phrase ``absolute
constant'' will refer to fixed numbers that are independent of everything.

\subsection*{Related work}

Our approach is motivated by recent advances in the study of the two
dimensional discrete Gaussian free field (DGFF). In \cite{bz}, Bramson
and Zeitouni computed the expected maximum of the DGFF up to an order
1 error and concluded tightness of the recentered maximum. In \cite{d},
Ding obtained bounds on the right and left tail of the recentered
maximum of the DGFF. Later on, in \cite{bdz}, Bramson, Ding and Zeitouni
proved convergence in distribution of the recentered maximum. The
approach of this line of research is to use first and second moment
methods, together with decomposition properties of the DGFF, to obtain
good estimates on tail events. Previous work on the DGFF includes
\cite{bdg}, where Bolthausen, Deuschel and Giacomin obtained asympotics
for the maximum of the DGFF, and \cite{da}, where Daviaud studied
the extreme points of the DGFF. On the other hand, previous work on
the continuous Gaussian free field (CGFF) includes \cite{hmp}, where
Hu, Miller, and Peres studied the Hausdorff dimension of the ``thick
points'' of the MGFF, which are closely related to the work of Daviaud.
We also mention \cite{dy} for a nice discussion of Gaussian fields
induced by Markov processes, and \cite{s} for a survey on the CGFF.

Our main result implies, in particular, an analog of \cite[Theorem 1.1]{bz}
for the MGFF. Our approach consists on extending the MBRW by Brownian
sheet, so that it is possible to compare the extended field with scaled
log-correlated continuous fields. Log-correlated Gaussian fields are
subject of current interest (see \cite{drsv}, \cite{m}, \cite{mrv}).
In particular, in \cite{m}, Madaule proved convergence for stationary
centered Gaussian fields $\left(Z_{\epsilon}(x):x\in[0,1]^{d}\right)$
whose covariance satisfies
\[
Cov(Z_{\epsilon}(0),Z_{\epsilon}(x))=\int_{0}^{\log(1/\epsilon)}k(e^{r}x)dr,
\]
where the fixed kernel $k:\mathbb{R}^{d}\to\mathbb{R}$ is of class
$C^{1}$, vanishes outside $[-1,1]^{d}$, and satisfies $k(0)=1$.
Theorem \ref{main} has weaker conditions on the covariance structure,
and consequently, only tightness is achieved. 

In \cite{mrv}, the authors proved the so called ``Freezing Theorem
for GFF in planar domains'' for a sequence of Gaussian fields approximating
the continuous GFF by cutting-off white noise, so that the covariance
kernel is proportional to the function $G_{t}:[0,1]^{2}\times[0,1]^{2}\to\mathbb{R}$
given by
\[
G_{t}(x,y)=\int_{e^{-t}}^{\infty}p_{\partial[0,1]^{2}}(r,x,y)dr,
\]
where $p_{\partial[0,1]^{2}}(r,x,y)$ is the transition probability
density of a Brownian motion killed at $\partial[0,1]^{2}$. In the
present paper, we consider a sequence of fields approximating the
GFF by mollifying the Green function (see \eqref{green}), and we
prove tightness. Convergence for the MGFF is expected to follow by
adapting of the arguments given in \cite{bdz}.

\section{Comparison to the MBRW}

\subsection{Auxiliary fields}

In this subsection, we rigorously introduce the fields we mentioned
in Section 1. A few properties of these fields will be stated; the
proofs of these properties will be given in the Appendix.

In order to define these fields, it will be notationally more convenient
to use $[0,1)^{d}$ instead of $[0,1]^{d}$ as the index set. This
will not affect the main result because the supremum of $Y_{\epsilon}$
over $[0,1)^{d}$ is the same, due to continuity, as the maximum over
$[0,1]^{d}$.

\subsubsection*{Modified branching random walk}

We first divide $[0,1)^{d}$ into boxes of side length $\epsilon>0$.
Let $V_{\epsilon}=\left(\epsilon\mathbb{Z}^{d}\right)\cap[0,1)^{d}$
and, for $v\in V_{\epsilon}$, let $\square_{\epsilon}^{v}=[v,v+\epsilon)^{d}\cap[0,1)^{d}$.
Moreover, if $x\in\square_{\epsilon}^{v}$, let $[x]:=v$. The set
$V_{\epsilon}$ is, of course, a discretized version of $[0,1)^{d}$. 

We now define the \emph{modified branching random walk} (MBRW) as
the centered Gaussian field $\left\{ \xi_{\epsilon}^{v}(t):v\in V_{\epsilon},0\leq t\leq\log(1/\epsilon)\right\} $
with covariance structure
\begin{equation}
Cov(\xi_{\epsilon}^{v}(t),\xi_{\epsilon}^{u}(s))=\int_{0}^{\min\left\{ t,s\right\} }\prod_{1\leq i\leq d}(1-e^{r}\left|v_{i}-u_{i}\right|)_{+}dr\label{defMBRW}
\end{equation}
for all $0\leq t,s\leq\log\left(1/\epsilon\right)$ and $v,u\in V_{\epsilon}$,
where $v_{i}$ is the $i$-th coordinate of $v$, and $\left(\cdot\right)_{+}=\max\left\{ \cdot,0\right\} $.
For simplicity, write $\xi_{\epsilon}^{v}=\xi_{\epsilon}^{v}(\log(1/\epsilon))$. 

Note that, for each point $v\in V_{\epsilon}$, the process $\left(\xi_{\epsilon}^{v}(t)\right)_{t}$
is a standard Brownian motion. Moreover, for each pair $v,u\in V_{\epsilon}$,
the Brownian motions are correlated until $t=-\log\left\Vert v-u\right\Vert _{\infty}$,
at which time their increments become independent. The end time is
$t=\log\left(1/\epsilon\right)$, because, for the ``usual'' $d$-ary
branching random walk, it takes $\log(1/\epsilon)$ units of time
to generate $\left|V_{\epsilon}\right|$ particles (see the proof
of Proposition \ref{proofrtubMBRW} for a definition of the usual
$d$-ary branching random walk).

It will be proved in the Appendix (see Proposition \ref{proofdefMBRW})
that the MBRW exists and that it satisfies
\begin{equation}
Var(\xi_{\epsilon}^{v})=\log(1/\epsilon)\label{varMBRW}
\end{equation}
and, for $v\neq u$ (so that $\left\Vert v-u\right\Vert _{\infty}\geq\epsilon$),
\begin{equation}
-\log\left\Vert v-u\right\Vert _{\infty}-C\leq Cov(\xi_{\epsilon}^{v},\xi_{\epsilon}^{u})\leq-\log\left\Vert v-u\right\Vert _{\infty}\label{covMBRW}
\end{equation}
for some constant $C$ depending on $d$. The MBRW also satisfies
(see Proposition \ref{proofrtlbMBRW}) 
\begin{equation}
\mathbb{P}\left(\max_{v\in V_{\epsilon}}\xi_{\epsilon}^{v}\geq m_{\epsilon}\right)\geq c>0,\label{rtlbMBRW}
\end{equation}
where $c$ is a constant depending only on $d$. It will also be proved
in the Appendix (see Proposition \ref{proofrtubMBRW}) that there
exist constants $0<c,C<\infty$ (depending on $d$) such that
\begin{equation}
\mathbb{P}\left(\max_{v\in A}\xi_{\epsilon}^{v}\geq m_{\epsilon}+z\right)\leq C\left(\epsilon^{d}\left|A\right|\right)^{1/2}e^{-cz}\label{rtubMBRW}
\end{equation}
for all $A\subset V_{\epsilon}$, $z\in\mathbb{R}$ and $\epsilon>0$
small enough, where $\left|A\right|$ is the cardinality of $A$.

\subsubsection*{Brownian sheet}

As mentioned before, we will need an additional continuous Gaussian
field. For $x=(x_{i})_{i\leq d}\in\mathbb{R}_{+}^{d}$, let $\psi^{x}$
denote the centered standard Brownian sheet. Recall that it satisfies

\[
\mathbb{E}\left[\psi^{x}\psi^{y}\right]=\prod_{i\leq d}\min\left\{ x_{i},y_{i}\right\} .
\]
Define a new field $\left(\psi_{\epsilon}^{x}:x\in[0,1)^{d}\right)$,
depending on a parameter $p\geq1$, as follows: for $v\in V_{\epsilon}$,
let $l$ be the linear map from $\square_{\epsilon}^{v}$ onto $[p,2p)^{d}$
sending $v$ to $(p)_{i\leq d}=(p,p,\ldots,p)$. Set

\begin{equation}
\left(\psi_{\epsilon}^{x}:x\in\square_{\epsilon}^{v}\right):\overset{d}{=}\left(\psi^{l(x)}:x\in\square_{\epsilon}^{v}\right)=\left(\psi^{x}:x\in[p,2p)^{d}\right),\label{defBS}
\end{equation}
for each $v\in V_{\epsilon}$, and choose $\psi_{\epsilon}^{x}$ and
$\psi_{\epsilon}^{y}$ to be independent if $[x]\neq[y]$. Note that
the collection of fields $\left\{ \left(\psi_{\epsilon}^{x}:x\in\square_{\epsilon}^{v}\right)\right\} _{v\in V_{\epsilon}}$
consist of i.i.d copies of Brownian sheet on $[p,2p)^{d}$. Using
the covariance structure of the Brownian sheet, it is not hard to
see that

\begin{equation}
p^{d}\leq Var\left(\psi_{\epsilon}^{x}\right)\leq(2p)^{d},\label{varBS}
\end{equation}
for all $x\in[0,1)^{d}$, and that

\begin{equation}
p^{d}\epsilon^{-1}\left\Vert x-y\right\Vert _{1}\leq\mathbb{E}\left[\left(\psi_{\epsilon}^{x}-\psi_{\epsilon}^{y}\right)^{2}\right]\leq(2p)^{d}\epsilon^{-1}\left\Vert x-y\right\Vert _{1},\label{l2distBS}
\end{equation}
for all $[x]=[y]$. Note that $p$ can be chosen as large as desired.

To understand the motivation behind the previous definitions, we invite
the reader to compare the bounds \eqref{logcorrY} and \eqref{l2distY}
with \eqref{covMBRW} and \eqref{l2distBS}, respectively. These bounds
will be used in the next sections.

We now proceed to the comparison of the right and left tail of the
maximum of the field $Y_{\epsilon}$ (which was defined in Section
1 and satisfies \eqref{logcorrY} and \eqref{l2distY}) and the maximum
of an appropriate combination of the fields $\xi_{\epsilon}$ and
$\psi_{\epsilon}$ (which will be specified in the next section).
Note that we will only use Brownian sheet when comparing the right
tail; for the left tail, we will compare directly the MBRW with the
field $Y_{\epsilon}$ on a discrete index set.

\subsection{The right tail}

Recall from Section 1 that the field $Y_{\epsilon}$ satisfies \eqref{logcorrY}
and \eqref{l2distY}, by definition. 
\begin{prop}
\label{rtcomparison} For $\epsilon>0$, let $\left(\xi_{\epsilon}^{v}:v\in V_{\epsilon}\right)$
and $\left(\psi_{\epsilon}^{x}:x\in[0,1)^{d}\right)$ be independent
fields, defined as in \eqref{defMBRW} and \eqref{defBS}, respectively.
Then, there exist $\delta>0$ small enough and $p$ large enough (depending
on $C_{Y}$ and $d$) such that

\[
\mathbb{P}\left(\sup_{x\in[0,1)^{d}}Y_{\delta\epsilon}^{\delta x}\geq\lambda\right)\leq\mathbb{P}\left(\sup_{x\in[0,1)^{d}}a(x)\xi_{\epsilon}^{[x]}+\psi_{\epsilon}^{x}\geq\lambda\right)
\]
for all $\epsilon>0$ and all $\lambda\in\mathbb{R}$, where $a(x):=\sqrt{\left(Var(Y_{\delta\epsilon}^{\delta x})-Var(\psi_{\epsilon}^{x})\right)/Var(\xi_{\epsilon}^{[x]})}$.\end{prop}
\begin{proof}
We check the hypotheses of Slepian's Lemma (see \cite[Theorem 2.2.1]{rf}).
The variance of the fields $Y_{\delta\epsilon}^{\delta x}$ and $a(x)\xi_{\epsilon}^{[x]}+\psi_{\epsilon}^{x}$
are equal by the definition of $a(x)$. We first choose $p$ so that
$a(x)\leq1$. Note that \eqref{logcorrY} and \eqref{varBS} imply
\[
a(x)^{2}=\frac{Var(Y_{\delta\epsilon}^{\delta x})-Var(\psi_{\epsilon}^{x})}{Var(\xi_{\epsilon}^{[x]})}\leq\frac{\log(1/\epsilon)+\log(1/\delta)+C_{Y}-p^{d}}{\log(1/\epsilon)},
\]
so, by choosing $p$ large enough (depending on $C_{Y}$, $d$ and
$\delta$), we obtain $a(x)\leq1$, for all $x$.

We now compare the covariance for points $x\neq y$, for which we
distinguish two cases:

1. $[x]=[y]$ (that is, $\square_{\epsilon}^{[x]}=\square_{\epsilon}^{[y]}$).
In this case, \eqref{l2distY} and \eqref{l2distBS} imply
\[
\mathbb{E}\left[\left(Y_{\delta\epsilon}^{\delta x}-Y_{\delta\epsilon}^{\delta y}\right)^{2}\right]\leq C_{Y}(\delta\epsilon)^{-1}\left\Vert \delta x-\delta y\right\Vert \leq p^{d}\epsilon^{-1}\left\Vert x-y\right\Vert _{1}\leq\mathbb{E}\left[\left(\psi_{\epsilon}^{x}-\psi_{\epsilon}^{y}\right)^{2}\right]
\]
\[
\leq\mathbb{E}\left[\left(a(x)\xi_{\epsilon}^{[x]}+\psi_{\epsilon}^{x}-a(y)\xi_{\epsilon}^{[y]}-\psi_{\epsilon}^{y}\right)^{2}\right]
\]
for $p$ large enough (depending on $C_{Y}$). The last inequality
is due the independence between $\xi_{\epsilon}$ and $\psi_{\epsilon}$.

2. $[x]\neq[y]$. In this case, we can apply \eqref{covMBRW} and
the independence between $\xi_{\epsilon}$, $\psi_{\epsilon}^{[x]}$
and $\psi_{\epsilon}^{[y]}$ to obtain
\[
Cov(a(x)\xi_{\epsilon}^{[x]}+\psi_{\epsilon}^{x},a(y)\xi_{\epsilon}^{[y]}+\psi_{\epsilon}^{y})\leq a(x)a(y)Cov(\xi_{\epsilon}^{[x]},\xi_{\epsilon}^{[y]})\leq a(x)a(y)\left(-\log\left\Vert [x]-[y]\right\Vert +C\right).
\]
But $a(x)a(y)\leq1$, so
\[
Cov(a(x)\xi_{\epsilon}^{[x]}+\psi_{\epsilon}^{x},a(y)\xi_{\epsilon}^{[y]}+\psi_{\epsilon}^{y})\leq-\log\left\Vert [x]-[y]\right\Vert +C.
\]
Note that $-\log\left\Vert [x]-[y]\right\Vert \leq-\log\max\left\{ \epsilon,\left\Vert x-y\right\Vert \right\} +C$.
Applying \eqref{logcorrY}, we obtain
\[
-\log\max\left\{ \epsilon,\left\Vert x-y\right\Vert \right\} +C\leq-\log\max\left\{ \delta\epsilon,\left\Vert \delta x-\delta y\right\Vert \right\} -C_{Y}\leq Cov(Y_{\delta\epsilon}^{\delta x},Y_{\delta\epsilon}^{\delta y})
\]
for some $\delta>0$ small enough (depending on $C_{Y}$ and $d$).
Proposition \ref{rtcomparison} follows now from Slepian's Lemma. 
\end{proof}
Proposition \ref{rtcomparison} provides an upper bound for the right
tail of the supremum of $Y_{\delta\epsilon}$ taken over the $\delta$-box
$\delta[0,1)^{d}$. The same proof works for any $\delta$-box. Therefore,
a union bound implies

\begin{equation}
\mathbb{P}\left(\sup_{x\in[0,1)^{d}}Y_{\delta\epsilon}^{x}\geq\lambda\right)\leq\left(\frac{1}{\delta}\right)^{d}\mathbb{P}\left(\sup_{x\in[0,1)^{d}}a(x)\xi_{\epsilon}^{[x]}+\psi_{\epsilon}^{x}\geq\lambda\right)\label{union}
\end{equation}
for all $\lambda\in\mathbb{R}$. 

We now provide an upper bound for the probability on the right hand
side of the previous display. We first prove an upper bound on the
supremum of the Brownian sheet.
\begin{lem}
\label{rtubBS} There exist constants $0<c,C<\infty$ (depending on
$p$ and $d$) such that

\[
\sup_{v\in V_{\epsilon}}\mathbb{P}\left(\sup_{x\in\square_{\epsilon}^{v}}\psi_{\epsilon}^{x}\geq\lambda\right)\leq Ce^{-c\lambda^{2}}
\]
for all $\lambda\geq0,\epsilon>0$.\end{lem}
\begin{proof}
Let $v\in V_{\epsilon}$. Fernique's Majorizing Criterion (see \cite[Theorem 4.1]{ce})
implies that

\[
\mathbb{E}\left[\sup_{x\in\square_{\epsilon}^{v}}\psi_{\epsilon}^{x}\right]\leq C\sup_{x\in\square_{\epsilon}^{v}}\int_{0}^{\infty}\sqrt{-\log\left(\mu(B(x,r))\right)}dr
\]
for some absolute constant $C$, where $\mu$ is the normalized $d$-dimensional
Lebesgue measure on $\square_{\epsilon}^{v}$ and $B(x,r)=\left\{ y\in\square_{\epsilon}^{v}:\mathbb{E}\left[\left(\psi_{\epsilon}^{x}-\psi_{\epsilon}^{y}\right)^{2}\right]\leq r^{2}\right\} $.
But \eqref{l2distBS} implies 
\[
B(x,r)\supset\left\{ y\in\square_{\epsilon}^{v}:(2p)^{d}\epsilon^{-1}\left\Vert y-x\right\Vert _{1}\leq r^{2}\right\} .
\]
Therefore, $\mu\left(B(x,r)\right)\geq cr^{2d}$ for some constant
$c>0$ depending on $p$ and $d$. Applying the previous display and
Fernique's Majorizing Criterion, we obtain

\[
\mathbb{E}\left[\sup_{x\in\square_{\epsilon}^{v}}\psi_{\epsilon}^{x}\right]\leq C\int_{0}^{\infty}\sqrt{-\log\left(cr^{2d}\right)}dr\leq C<\infty,
\]
where $C$ depends on $p$ and $d$. Borell's Inequality (see \cite[Theorem 2.1.1]{rf})
and \eqref{varBS} imply

\[
\mathbb{P}\left(\sup_{x\in\square_{\epsilon}^{v}}\psi_{\epsilon}^{x}\geq C+\lambda\right)\leq e^{-\lambda^{2}/(2(2p)^{d})},
\]
where $C$ is the constant obtained in the previous display. Lemma
\ref{rtubBS} now follows from a change of variables.\end{proof}
\begin{prop}
\label{rtcomparison2} Let $p$ be as in Proposition \ref{rtcomparison}.
There exist constants $0<c,C<\infty$ (depending on $p$ and $d$)
such that

\[
\mathbb{P}\left(\sup_{x\in[0,1)^{d}}a(x)\xi_{\epsilon}^{[x]}+\psi_{\epsilon}^{x}\geq\lambda+m_{\epsilon}\right)\leq Ce^{-c\lambda}
\]
for all $\lambda\geq0$ and all $\epsilon>0$ small enough.\end{prop}
\begin{proof}
By letting $\psi_{\epsilon}^{\ast,[x]}=\sup_{y\in\square_{\epsilon}^{[x]}}\psi_{\epsilon}^{y}$,
we have
\[
\sup_{x\in[0,1)^{d}}a(x)\xi_{\epsilon}^{[x]}+\psi_{\epsilon}^{x}\leq\max_{x\in[0,1)^{d}}a(x)\xi_{\epsilon}^{[x]}+\psi_{\epsilon}^{\ast,[x]}.
\]
The previous display implies

\[
\sup_{x\in[0,1)^{d}}a(x)\xi_{\epsilon}^{[x]}+\psi_{\epsilon}^{x}\geq m_{\epsilon}+\lambda\implies\sup_{x\in[0,1)^{d}}a(x)\xi_{\epsilon}^{[x]}+\psi_{\epsilon}^{\ast,[x]}\geq m_{\epsilon}+\lambda.
\]

We now compute an upper bound for the right hand side of the previous
display. Define the random sets $\Gamma_{y}=\left\{ v\in V_{\epsilon}:\psi_{\epsilon}^{\ast,v}\in[y-1,y)\right\} $
for $y\geq1$, and $\Gamma_{0}=\left\{ v\in V_{\epsilon}:\psi_{\epsilon}^{\ast,v}\leq0\right\} $.
Note that
\[
\mathbb{P}\left(\sup_{x\in[0,1)^{d}}a(x)\xi_{\epsilon}^{[x]}+\psi_{\epsilon}^{\ast,[x]}\geq m_{\epsilon}+\lambda\right)\leq\sum_{y\geq0}\mathbb{P}\left(\sup_{x:[x]\in\Gamma_{y}}a(x)\xi_{\epsilon}^{[x]}\geq m_{\epsilon}+\lambda-y\right).
\]
The definition of $a(x)$ easily implies that $1/2\leq a(x)\leq1$,
for $\epsilon>0$ small enough. Therefore, the last display implies
\begin{equation}
\mathbb{P}\left(\sup_{x\in[0,1)^{d}}a(x)\xi_{\epsilon}^{[x]}+\psi_{\epsilon}^{\ast,[x]}\geq m_{\epsilon}+\lambda\right)\leq\sum_{y\geq0}\mathbb{P}\left(\max_{v\in\Gamma_{y}}\xi_{\epsilon}^{v}\geq m_{\epsilon}+\lambda-2y\right).\label{randomsets}
\end{equation}
But $\mathbb{P}\left(\max_{v\in\Gamma_{y}}\xi_{\epsilon}^{v}\geq m_{\epsilon}+\lambda-2y\right)=\mathbb{E}\left[\mathbb{P}\left(\max_{v\in\Gamma_{y}}\xi_{\epsilon}^{v}\geq m_{\epsilon}+\lambda-2y\mid\Gamma_{y}\right)\right]$.
Since $\psi_{\epsilon}$ and $\xi_{\epsilon}$ are independent, from
\eqref{rtubMBRW} we obtain, 
\[
\mathbb{P}\left(\max_{v\in\Gamma_{y}}\xi_{\epsilon}^{v}\geq m_{\epsilon}+\lambda-2y\mid\Gamma_{y}\right)\leq C\left(\epsilon^{d}\left|\Gamma_{y}\right|\right)^{1/2}e^{-c(\lambda-2y)}.
\]
Then,
\begin{equation}
\mathbb{P}\left(\max_{v\in\Gamma_{y}}\xi_{\epsilon}^{v}\geq m_{\epsilon}+\lambda-2y\right)\leq Ce^{-c(\lambda-2y)}\left(\mathbb{E}\left[\epsilon^{d}\left|\Gamma_{y}\right|\right]\right)^{1/2}.\label{randomsets2}
\end{equation}
But, by Lemma \ref{rtubBS}, $\mathbb{E}\left[\left|\Gamma_{y}\right|\right]=\sum_{v\in V_{\epsilon}}\mathbb{P}\left(\psi_{\epsilon}^{\ast,v}\in[y-1,y)\right)\leq C\epsilon^{-d}e^{-cy^{2}}$.
For $y=0$, we simply use $\left|\Gamma_{0}\right|\leq\epsilon^{-d}$.
Therefore, from displays \eqref{randomsets} and \eqref{randomsets2},
we obtain
\[
\mathbb{P}\left(\sup_{x\in[0,1)^{d}}a(x)\xi_{\epsilon}^{[x]}+\psi_{\epsilon}^{\ast,[x]}\geq m_{\epsilon}+\lambda\right)\leq Ce^{-c\lambda}
\]
for some constants $0<c,C<\infty$ (depending on $p$ and $d$). 
\end{proof}

\begin{proof}
[Proof of Theorem \ref{main}, \eqref{rtubY}, the right tail] Display
\eqref{union} and Proposition \ref{rtcomparison2} imply

\[
\mathbb{P}\left(\max_{x\in[0,1)^{d}}Y_{\delta\epsilon}^{x}\geq m_{\epsilon}+\lambda\right)\leq\left(\frac{1}{\delta}\right)^{2}\mathbb{P}\left(\max_{x\in[0,1)^{d}}a(x)\xi_{\epsilon}^{[x]}+\psi_{\epsilon}^{x}\geq\lambda+m_{\epsilon}\right)\leq Ce^{-c\lambda}.
\]
It is easy to see from the definition that $m_{\delta\epsilon}\leq m_{\epsilon}+C'$
for some $C'$ depending on $\delta$ and $d$. Therefore,

\[
\mathbb{P}\left(\max_{x\in[0,1)^{d}}Y_{\delta\epsilon}^{x}\geq m_{\delta\epsilon}+\lambda-C'\right)\leq Ce^{-c\lambda}.
\]
The upper bound \eqref{rtubY} for the right tail follows by adjusting
the constants.
\end{proof}

\subsection{The left tail}

In this subsection we prove the upper bound \eqref{rtubY} for the
left tail. As previously mentioned, we can reduce the set under maximization
to a discrete set. More precisely, if $\left\{ D_{\epsilon}:\epsilon>0\right\} $
is any collection of subsets of $[0,1)^{d}$, then

\begin{equation}
\mathbb{P}\left(\sup_{x\in[0,1)^{d}}Y_{\epsilon}^{x}\leq m_{\epsilon}-\lambda\right)\leq\mathbb{P}\left(\sup_{x\in D_{\epsilon}}Y_{\epsilon}^{x}\leq m_{\epsilon}-\lambda\right).\label{discrete}
\end{equation}
If we select $D_{\epsilon}$ appropriately, we can perform a comparison
with the MBRW using Slepian's Lemma.
\begin{prop}
\label{ltcomparison} There exist $\delta,\rho>0$ small enough (depending
on $C_{Y}$ and $d$) such that

\[
\mathbb{P}\left(\max_{u\in V_{\epsilon/\rho}}Y_{\delta\epsilon}^{u}\leq\lambda\right)\leq\mathbb{P}\left(\max_{u\in V_{\epsilon}\cap\rho[0,1)^{d}}b(u)\xi_{\epsilon}^{u}\leq\lambda\right)
\]
for all $\epsilon>0$ and all $\lambda\in\mathbb{R}$, where $b(u):=\sqrt{Var(Y_{\delta\epsilon}^{u})/Var(\xi_{\epsilon}^{u})}$,
for $u\in V_{\epsilon/\rho}$.\end{prop}
\begin{proof}
Note that \eqref{logcorrY} and \eqref{covMBRW} imply that $b(u)\geq\frac{\log(1/\epsilon)+\log(1/\delta)-C_{Y}}{\log(1/\epsilon)}$,
which is greater than 1 for $\delta>0$ small enough (depending on
$C_{Y}$).

Let $u,v\in V_{\epsilon/\rho}$, with $u\neq v$. Then, $\left\Vert u-v\right\Vert \geq\epsilon/\rho\geq\delta\epsilon$.
Display \eqref{logcorrY} therefore implies

\[
Cov(Y_{\delta\epsilon}^{u},Y_{\delta\epsilon}^{v})\leq-\log\left\Vert u-v\right\Vert +C_{Y}.
\]
Choose $\rho>0$ small enough so that

\[
-\log\left\Vert u-v\right\Vert +C_{Y}\leq-\log\left\Vert \rho u-\rho v\right\Vert -C\leq Cov(\xi_{\epsilon}^{\rho u},\xi_{\epsilon}^{\rho v}),
\]
where the last bound follows from \eqref{covMBRW}. All the hypotheses
of Slepian's Lemma are satisfied, so 

\[
\mathbb{P}\left(\max_{u\in V_{\epsilon/\rho}}Y_{\delta\epsilon}^{u}\leq\lambda\right)\leq\mathbb{P}\left(\max_{u\in V_{\epsilon/\rho}}b(u)\xi_{\epsilon}^{\rho u}\leq\lambda\right)
\]
for all $\lambda\in\mathbb{R}$. Proposition \ref{ltcomparison} follows
by observing that $\rho V_{\epsilon/\rho}=V_{\epsilon}\cap\rho[0,1]^{2}$.\end{proof}
\begin{prop}
\label{ltcomparison2} Let $\rho>0$ and $\left\{ b(u):u\in V_{\epsilon}\cap\rho[0,1]^{2}\right\} $
be as in Proposition \ref{ltcomparison}. Then,

\[
\mathbb{P}\left(\max_{u\in V_{\epsilon}\cap\rho[0,1]^{2}}b(u)\xi_{\epsilon}^{u}\leq m_{\epsilon}-\lambda\right)\leq\mathbb{P}\left(\max_{u\in V_{\epsilon}\cap\rho[0,1]^{2}}\xi_{\epsilon}^{u}\leq m_{\epsilon}-\lambda/2\right)
\]
for all $\lambda\geq0$ and all $\epsilon>0$ small enough.\end{prop}
\begin{proof}
It follows from the definition of $b(u)$ that, for small enough $\epsilon>0$,

\[
1\leq b(u)\leq2
\]
for all $u$. Let $\nu$ be the (a.s. well-defined) point that maximizes
$\xi_{\epsilon}^{u}$, for $u\in V_{\epsilon}\cap\rho[0,1]^{2}$.
Then,

\[
b(\nu)\xi_{\epsilon}^{\nu}\leq m_{\epsilon}-\lambda\implies\xi_{\epsilon}^{\nu}\leq m_{\epsilon}/b(\nu)-\lambda/b(\nu)\leq m_{\epsilon}-\lambda/2.
\]

\end{proof}
Our task is now to find an upper bound for the probability on the
right hand side of Proposition \ref{ltcomparison2}.
\begin{prop}
\label{ltcompare3} There exist constants $0<c,C<\infty$ (depending
on $\rho$ and $d$) such that

\[
\mathbb{P}\left(\max_{v\in V_{\epsilon}\cap\rho[0,1]^{2}}\xi_{\epsilon}^{v}\leq m_{\epsilon}-\lambda\right)\leq Ce^{-c\lambda}
\]
for all $\lambda\geq0$ and all $\epsilon>0$ small enough.\end{prop}
\begin{proof}
Assume $0\leq k\leq\log\left(1/\epsilon\right)/2$, where $k$ is
a large number, that will be chosen later. Let $\left\{ B^{i}:i=1,\ldots,ce^{k}\right\} $
(where $c>0$ is a small constant, depending on $\rho$) be a collection
of boxes of side length $e^{-k}$ inside $\rho[0,1)^{d}$, such that
the distance between any pair of boxes is at least $e^{-k}$. Set
$B_{\epsilon}^{i}=B^{i}\cap V_{\epsilon}$. We claim that the field
\[
\left(\xi_{\epsilon}^{v}-\xi_{\epsilon}^{v}(k):v\in B_{\epsilon}^{i}\right)
\]
is a copy of $\left(\xi_{\epsilon e^{k}}^{v}:v\in V_{\epsilon e^{k}}\right)$,
and that the fields $\left\{ \left(\xi_{\epsilon}^{v}-\xi_{\epsilon}^{v}(k):v\in B_{\epsilon}^{i}\right)\right\} _{i\leq ce^{k}}$
are independent. Indeed, if $v,u\in B_{\epsilon}^{i}$, then \eqref{defMBRW}
implies
\begin{equation}
Cov(\xi_{\epsilon}^{v}-\xi_{\epsilon}^{v}(k),\xi_{\epsilon}^{u}-\xi_{\epsilon}^{u}(k))=\int_{k}^{\log(1/\epsilon)}\prod_{j\leq d}\left(1-e^{r}\left|v_{j}-u_{j}\right|\right)_{+}dr\label{copies}
\end{equation}
\[
=\int_{0}^{-\log(\epsilon e^{k})}\prod_{j\leq d}\left(1-e^{r}\left|e^{k}v_{j}-e^{k}u_{j}\right|\right)_{+}dr,
\]
and the set $e^{k}B_{\epsilon}^{i}=\left\{ e^{k}v:v\in B_{\epsilon}^{i}\right\} $
coincides with $V_{\epsilon e^{k}}$ after a translation. This shows
that $\left(\xi_{\epsilon}^{v}-\xi_{\epsilon}^{v}(k):v\in B_{\epsilon}^{i}\right)\overset{d}{=}\left(\xi_{\epsilon e^{k}}^{v}:v\in V_{\epsilon e^{k}}\right)$.
Moreover, from \eqref{copies}, it is easy to see that $\left\Vert v-u\right\Vert \geq e^{-k}$
(which is true for points $v,u$ in different boxes $B_{\epsilon}^{i}$,
by construction) implies
\[
Cov(\xi_{\epsilon}^{v}-\xi_{\epsilon}^{v}(k),\xi_{\epsilon}^{u}-\xi_{\epsilon}^{u}(k))=0,
\]
as desired. 

Therefore, independence of the fields $\left\{ \left(\xi_{\epsilon}^{v}-\xi_{\epsilon}^{v}(k):v\in B_{\epsilon}^{i}\right)\right\} _{i\leq ce^{k}}$
and \eqref{rtlbMBRW} imply
\[
\mathbb{P}\left(\max_{v\in\bigcup_{i}B_{\epsilon}^{i}}\left(\xi_{\epsilon}^{v}-\xi_{\epsilon}^{v}(k)\right)\leq m_{\epsilon e^{k}}\right)\leq e^{-ce^{k}}
\]
for some constant $c>0$ depending on $d$ and $\rho$. By letting
$\nu=\arg\max\left\{ \xi_{\epsilon}^{v}-\xi_{\epsilon}^{v}(k):v\in\bigcup_{i}B_{\epsilon}^{i}\right\} $,
the previous display implies 
\[
\mathbb{P}\left(\max_{v\in V_{\epsilon}\cap\rho[0,1)^{d}}\xi_{\epsilon}^{v}\leq m_{\epsilon}-\lambda\right)\leq\mathbb{P}\left(\xi_{\epsilon}^{\nu}\leq m_{\epsilon}-\lambda\right)\leq\mathbb{P}\left(\xi_{\epsilon}^{\nu}(k)\leq m_{\epsilon}-m_{\epsilon e^{k}}-\lambda\right)+\mathbb{P}\left(\xi_{\epsilon}^{\nu}-\xi_{\epsilon}^{\nu}(k)\leq m_{\epsilon e^{k}}\right)
\]
\[
\leq\mathbb{P}\left(\xi_{\epsilon}^{\nu}(k)\leq m_{\epsilon}-m_{\epsilon e^{k}}-\lambda\right)+e^{-ce^{k}}.
\]
Moreover, it is clear from \eqref{defMBRW} that the fields $\left(\xi_{\epsilon}^{v}-\xi_{\epsilon}^{v}(k):v\in V_{\epsilon}\right)$
and $\left(\xi_{\epsilon}^{v}(k):v\in V_{\epsilon}\right)$ are independent.
Hence, $\nu$ is independent from $\xi_{\epsilon}^{(\cdot)}(k)$,
and $\xi_{\epsilon}^{\nu}(k)$ is therefore a Gaussian random variable
with mean zero and variance $k$. But
\[
m_{\epsilon}-m_{\epsilon e^{k}}\leq\sqrt{2d}k.
\]
Therefore, by choosing $k=\log\lambda$, the last two displays imply
\[
\mathbb{P}\left(\max_{v\in V_{\epsilon}\cap\rho[0,1)^{d}}\xi_{\epsilon}^{v}\leq m_{\epsilon}-\lambda\right)\leq Ce^{-c\frac{\left(\lambda-\sqrt{2d}\log\lambda\right)^{2}}{\log\lambda}}+e^{-c\lambda}\leq Ce^{-c\lambda},
\]
proving Proposition \ref{ltcompare3} in the case $k=\log\lambda\leq\log(1/\epsilon)/2$. 

On the other hand, for $\lambda\geq\sqrt{1/\epsilon}$,
\[
\mathbb{P}\left(\max_{v\in V_{\epsilon}\cap\rho[0,1)^{d}}\xi_{\epsilon}^{v}\leq m_{\epsilon}-\lambda\right)\leq\mathbb{P}\left(\xi_{\epsilon}^{v}\leq m_{\epsilon}-\lambda\right)\leq Ce^{-c\frac{\left(\lambda-m_{\epsilon}\right)^{2}}{\log\left(1/\epsilon\right)}}\leq Ce^{-c\lambda}
\]
(where $v$ is any point), which implies Proposition \ref{ltcompare3}
in this case.
\end{proof}
Using Propositions \ref{ltcomparison}, \ref{ltcomparison2} and \ref{ltcompare3},
we are now ready to finish the proof of Theorem \ref{main}.
\begin{proof}
[Proof of \ref{main}, \eqref{rtubY}, the left tail] Propositions
\ref{ltcomparison}, \ref{ltcomparison2} and \ref{ltcompare3} imply
the existence of constants $0<\delta,\rho,c,C<\infty$, depending
on $C_{Y}$ and $d$, such that

\[
\mathbb{P}\left(\max_{u\in V_{\epsilon/\rho}}Y_{\delta\epsilon}^{u}\leq m_{\epsilon}-\lambda\right)\leq Ce^{-c\lambda}
\]
for all $\lambda\geq0$. But $m_{\delta\epsilon}\leq m_{\epsilon}+C'$,
where $C'$ depends on $\delta$ and $d$. Therefore,

\[
\mathbb{P}\left(\max_{u\in V_{\epsilon/\rho}}Y_{\delta\epsilon}^{u}\leq m_{\delta\epsilon}-\lambda-C'\right)\leq Ce^{-c\lambda}.
\]
The bound \eqref{rtubY} for the left tail follows by adjusting the
constants.
\end{proof}

\section{Example: a mollified Gaussian free field in $d=2$}

The Gaussian free field in two dimensions provides an important example
of a log-correlated field. Intuitively speaking, the reason for the
log-correlation is simply that, in $d=2$, the Green function for
the Laplacian is logarithmic.

We begin by recalling in Section 3.1 the definitions of the Dirichlet
product and the Hilbert space induced by it. We then use this Hilbert
space to define the continuous Gaussian free field and the mollified
Gaussian free field. After that, we prove some useful properties of
these fields, which will be used to check the hypotheses of Theorem
\ref{main}. Finally, in section 3.2, we use Theorem \ref{main} to
prove tightness of the recentered maximum of the family of mollified
Gaussian free fields.

\subsection{Continuous and mollified Gaussian free fields}

\subsubsection*{Dirichlet product}

We begin by recalling the definition of the Dirichlet product. Let
$C_{c}^{\infty}\left((0,1)^{2}\right)$ denote the set of real valued
$C^{\infty}$ functions with compact support in $(0,1)^{2}$. For
$\phi,\psi\in C_{c}^{\infty}\left((0,1)^{2}\right)$, let

\[
\langle\phi,\psi\rangle_{\nabla}=\int\nabla\phi(x)\nabla\psi(x)dx
\]
denote the Dirichlet product, where $\nabla$ is the gradient and
$dx$ is two-dimensional Lebesgue measure. Note that the Dirichlet
product satisfies

\begin{equation}
\langle\phi,\psi\rangle_{\nabla}=\int\phi(x)(-\Delta\psi)(x)dx,\label{parts}
\end{equation}
where $\Delta$ is the standard Laplacian. The Dirichlet product induces
a norm on $C_{c}^{\infty}\left((0,1)^{2}\right)$ by

\[
\left\Vert \phi\right\Vert _{\nabla}=\sqrt{\langle\phi,\phi\rangle_{\nabla}},
\]
called the Dirichlet norm. Denote by $W=W\left((0,1)^{2}\right)$
the completion of $C_{c}^{\infty}\left((0,1)^{2}\right)$ with respect
to the Dirichlet norm. The set $W$, together with the Dirichlet product
on $W$, defines a Hilbert space. 

The Dirichlet norm satisfies Poincare's Inequality: there exists a
constant $C$ (which depends only on the domain $(0,1)^{2}$) such
that

\[
\left\Vert \phi\right\Vert _{L^{2}}\leq C\left\Vert \nabla\phi\right\Vert _{L^{2}}
\]
for all $\phi\in C_{c}^{\infty}$. Poincare's Inequality implies that
the Dirichlet norm is equivalent to the norm

\[
\left\Vert \phi\right\Vert _{L^{2}}+\left\Vert \frac{\partial}{\partial x_{1}}\phi\right\Vert _{L^{2}}+\left\Vert \frac{\partial}{\partial x_{2}}\phi\right\Vert _{L^{2}}.
\]
Recall that the completion of $C_{c}^{\infty}\left((0,1)^{2}\right)$
with respect to the latter norm is called a $(1,2)$-Sobolev space
(i.e., measurable functions such that their weak derivatives up to
order 1 exist and belong to $L^{2}\left((0,1)^{2}\right))$. Since
the norms are equivalent, the space $\left(W,\left\Vert \cdot\right\Vert _{\nabla}\right)$
is also a Sobolev space. Therefore, for any $g\in W$ and any measurable
set $E\subset[0,1]^{2}$, the integral $\int_{E}g(x)dx$ is well-defined. 

For a given open set $U\subset(0,1)^{2}$, Poincare's Inequality implies
that the linear mapping $W\to\mathbb{R}$ given by

\[
g\mapsto\int_{U}g(x)dx
\]
is $\left\Vert \cdot\right\Vert _{\nabla}$-continuous. Note that,
since $W$ is a Hilbert space, the Riesz representation theorem implies
the existence of a function $f=f_{U}\in W$ such that

\begin{equation}
\langle g,f_{U}\rangle_{\nabla}=\int_{U}g(x)dx\label{riesz}
\end{equation}
for all $g\in W$.

\subsubsection*{Gaussian free fields}

The continuous Gaussian free field is defined as follows: since $\langle\cdot,\cdot\rangle_{\nabla}$
is positive definite, there exists a family $\left\{ X^{f}:f\in W\right\} $
of centered Gaussian variables, defined on some probability space
$\left(\Omega,\mathbb{P}\right)$, such that

\[
Cov(X^{f},X^{g})=\langle f,g\rangle_{\nabla}.
\]
The family $\left\{ X^{f}:f\in W\right\} $ is called the \emph{continuous
Gaussian free field}. 

We next define a field indexed by the set $[0,1]^{2}$. Fix $\epsilon>0$,
and let $x\in[0,1]^{2}$. By \eqref{riesz}, there exists a function
$f_{x,\epsilon}\in W$ such that

\begin{equation}
\langle f_{x,\epsilon},g\rangle_{\nabla}=\frac{1}{\pi\epsilon^{2}}\int_{D(x,\epsilon)\cap(0,1)^{2}}g(u)du\label{def}
\end{equation}
for all $g\in W$, where $D(x,\epsilon)$ is the disk of radius $\epsilon$
centered at $x$. Using \eqref{parts} and \eqref{def}, it is not
hard to show that

\begin{equation}
f_{x,\epsilon}(u)=\frac{1}{\pi\epsilon^{2}}\int_{D(x,\epsilon)\cap(0,1)^{2}}G(u,v)dv,\label{GRRREN}
\end{equation}
where $G=G_{(0,1)^{2}}$ is the Green function of $\left(0,1\right)^{2}$
for the operator $-\Delta$, with Dirichlet boundary conditions on
$\partial\left(0,1\right)^{2}$. For the domain $\left(0,1\right)^{2}$,
the Green function can be explicitly stated as:
\[
G(u,v)=\frac{4}{\pi^{2}}\sum_{n,m\geq1}\frac{1}{n^{2}+m^{2}}\sin\left(n\pi u_{1}\right)\sin\left(m\pi u_{2}\right)\sin\left(n\pi v_{1}\right)\sin\left(m\pi v_{2}\right),
\]
where $u=(u_{1},u_{2})\in[0,1]^{2}$. The field $\left(X^{f_{x,\epsilon}}:x\in[0,1]^{2}\right)$
will be called \emph{$\epsilon$-mollified Gaussian free field} (MGFF).
To simplify notation, set $X_{\epsilon}^{x}=X^{f_{x,\epsilon}}$.
Note that, by definition,

\[
Cov(X_{\epsilon}^{x},X_{\epsilon}^{y})=\langle f_{x,\epsilon},f_{y,\epsilon}\rangle_{\nabla}=\frac{1}{\pi\epsilon^{2}}\int_{D(x,\epsilon)\cap(0,1)^{2}}f_{y,\epsilon}(u)du
\]
and, from \eqref{GRRREN}, we obtain 

\begin{equation}
Cov(X_{\epsilon}^{x},X_{\epsilon}^{y})=\frac{1}{\left(\pi\epsilon^{2}\right)^{2}}\int_{D(x,\epsilon)\cap(0,1)^{2}\times D(y,\epsilon)\cap(0,1)^{2}}G(u,v)\, dudv,\label{green}
\end{equation}
for all $x,y\in[0,1]^{2}$.

\subsubsection*{Orthogonal decomposition}

The next proposition shows that the MGFF satisfies a tree-like decomposition
property.
\begin{prop}
\label{prop3.1} Let $Q=\frac{1}{2}(0,1)^{2}\subset(0,1)^{2}$ be
a sub-square of side length $1/2$. Then, $X_{\epsilon}^{x}$ can
be decomposed as

\[
X_{\epsilon}^{x}=\hat{X}_{\epsilon}^{x}+\phi^{x},
\]
where $\left(\hat{X}_{\epsilon}^{x}:x\in\overline{Q}\right)$ is a
copy of $\left(X_{2\epsilon}^{x}:x\in[0,1]^{2}\right)$, and $\left(\hat{X}_{\epsilon}^{x}:x\in\overline{Q}\right)$
is independent of $\left(\phi^{x}:x\in[0,1]^{2}\right)$.\end{prop}
\begin{proof}
Denote by $C_{c}^{\infty}\left(Q\right)$ the set of real valued $C^{\infty}$
functions with compact support in $Q$, and let $W(Q)$ be the corresponding
Hilbert space induced by the Dirichlet product in $C_{c}^{\infty}\left(Q\right)$.
Note that $C_{c}^{\infty}\left(Q\right)\subset C_{c}^{\infty}\left((0,1)^{2}\right)$
and 

\begin{equation}
\langle f,g\rangle_{\nabla,Q}:=\int_{Q}\nabla f(u)\cdot\nabla g(u)\, du=\int_{(0,1)^{2}}\nabla f(u)\cdot\nabla g(u)\, du\label{isom}
\end{equation}
for all $f,g\in C_{c}^{\infty}(Q)$. By taking the completion of $C_{c}^{\infty}\left(Q\right)$
with respect to the Dirichlet product, we see that $W(Q)$ is a Hilbert
subspace of $W\left((0,1)^{2}\right)$ and that \eqref{isom} holds
for all $f,g\in W(Q)$.

Let $f_{x,\epsilon}$ be as in \eqref{def} and decompose it as

\[
f_{x,\epsilon}=g_{x,\epsilon}+h_{x,\epsilon},
\]
where $g_{x,\epsilon}\in W(Q)$ and $h_{x,\epsilon}\in W(Q)^{\perp}$
(the orthogonal space). Set

\[
\hat{X}_{\epsilon}^{x}=X^{g_{x,\epsilon}}
\]
and

\[
\phi^{x}=X^{h_{x,\epsilon}}.
\]
Since $g_{x,\epsilon}\perp h_{y,\epsilon}$ for all $x,y\in[0,1]^{2}$,
the families $\left(\hat{X}_{\epsilon}^{x}:x\in[0,1]^{2}\right)$
and $\left(\phi^{x}:x\in[0,1]^{2}\right)$ are independent. Also,
since $f\mapsto X^{f}$ is a linear embedding of $W$ into $L^{2}\left(\Omega,\mathbb{P}\right)$, 

\[
X_{\epsilon}^{x}=\hat{X}_{\epsilon}^{x}+\phi^{x}\text{ a.s.}
\]
for every $x\in[0,1]^{2}$. 

We show now that $\left(\hat{X}_{\epsilon}^{x}:x\in\overline{Q}\right)$
is a copy of $\left(X_{2\epsilon}^{x}:x\in[0,1]^{2}\right)$.
\begin{claim}
For every $k\in W(Q)$, 

\[
\langle g_{x,\epsilon},k\rangle_{\nabla,Q}=\frac{1}{\pi\epsilon^{2}}\int_{D(x,\epsilon)\cap Q}k(u)\, du.
\]
\end{claim}
\begin{proof}
[Proof of Claim 3.2] By \eqref{isom},

\[
\langle g_{x,\epsilon},k\rangle_{\nabla,Q}=\langle g_{x,\epsilon},k\rangle_{\nabla}
\]
and, since $g_{x,\epsilon}=f_{x,\epsilon}-h_{x,\epsilon}$,

\[
\langle g_{x,\epsilon},k\rangle_{\nabla}=\langle f_{x,\epsilon},k\rangle_{\nabla}-\langle h_{x,\epsilon},k\rangle_{\nabla}.
\]
But $h_{x,\epsilon}\perp k$, so the second term on the right hand
side of the previous display vanishes. Using \eqref{def} and the
two previous displays, we obtain

\[
\langle g_{x,\epsilon},k\rangle_{\nabla,Q}=\frac{1}{\pi\epsilon^{2}}\int_{D(x,\epsilon)\cap(0,1)^{2}}k(u)\, du.
\]
Since $k\in W(Q)$, the function $k$ vanishes outside of $Q$. Therefore,

\[
\langle g_{x,\epsilon},k\rangle_{\nabla,Q}=\frac{1}{\pi\epsilon^{2}}\int_{D(x,\epsilon)\cap Q}k(u)\, du,
\]
as desired. 
\end{proof}
Claim 3.2 implies, in analogy with \eqref{green}, that the following
is true for all $x,y\in\overline{Q}$:

\[
Cov\left(\hat{X}_{\epsilon}^{x},\hat{X}_{\epsilon}^{y}\right)=\langle g_{x,\epsilon},g_{y,\epsilon}\rangle_{\nabla}=\langle g_{x,\epsilon},g_{y,\epsilon}\rangle_{\nabla,Q}=\frac{1}{\left(\pi\epsilon^{2}\right)^{2}}\int_{D(x,\epsilon)\cap Q\times D(y,\epsilon)\cap Q}G_{Q}(u,v)\, dudv,
\]
where $G_{Q}$ is the Green function of $Q$ for the operator $-\Delta$,
with Dirichlet boundary conditions on $\partial Q$. 
\begin{claim}
For every $u,v\in[0,1]^{2}$,

\[
G_{Q}\left(u/2,v/2\right)=G(u,v).
\]
\end{claim}
\begin{proof}
[Proof of Claim 3.3]Let $\phi\in C_{c}^{\infty}\left(\left(0,1\right)^{2}\right)$
and note that $\left(\Delta\phi\right)(2u)=\frac{1}{4}\Delta\left(\phi(2u)\right)$.
By the change of variables $u'=u/2$, 

\[
\int_{(0,1)^{2}}G_{Q}\left(u/2,v/2\right)\left(\Delta\phi\right)(u)\, du=\int_{Q}G_{Q}\left(u',v/2\right)\Delta\left(\phi\left(2u'\right)\right)\, du'=-\phi(2v/2),
\]
where the last equality holds by definition of $G_{Q}$. On the other
hand, 

\[
\int_{(0,1)^{2}}G(u,v)\left(\Delta\phi\right)(u)\, du=-\phi(v),
\]
by definition of $G$. Since

\[
\int_{(0,1)^{2}}G_{Q}\left(u/2,v/2\right)\left(\Delta\phi\right)(u)\, du=\int_{(0,1)^{2}}G(u,v)\left(\Delta\phi\right)(u)\, du
\]
for every $\phi\in C_{c}^{\infty}\left((0,1)^{2}\right)$, the functions
$G_{Q}(u/2,v/2)$ and $G(u,v)$ are identical (Lebesgue-a.e.).
\end{proof}
The change of variables $u'=2u,v'=2v$ implies

\[
Cov\left(\hat{X}_{\epsilon}^{x},\hat{X}_{\epsilon}^{y}\right)=\frac{1}{\left(\pi\epsilon^{2}\right)^{2}}\int_{D(x,\epsilon)\cap Q\times D(y,\epsilon)\cap Q}G_{Q}(u,v)\, dudv
\]

\[
=\frac{1}{\left(\pi(2\epsilon)^{2}\right)^{2}}\int_{D(2x,2\epsilon)\cap(0,1)^{2}\times D(2y,2\epsilon)\cap(0,1)^{2}}G_{Q}(u'/2,v'/2)\, du'dv',
\]
and Claim 3.3 implies that the previous display is

\[
=\frac{1}{\left(\pi(2\epsilon)^{2}\right)^{2}}\int_{D(2x,2\epsilon)\cap(0,1)^{2}\times D(2y,2\epsilon)\cap(0,1)^{2}}G(u',v')\, du'dv'=Cov\left(X_{2\epsilon}^{2x},X_{2\epsilon}^{2y}\right).
\]

For Gaussian fields, equality of the covariance structure implies
that the fields have the same distribution. Therefore,

\[
\left(\hat{X}_{\epsilon}^{x}:x\in\overline{Q}\right)\overset{d}{=}\left(X_{2\epsilon}^{2x}:x\in\overline{Q}\right),
\]
and the right hand side is clearly equal to $\left(X_{2\epsilon}^{x}:x\in[0,1]^{2}\right)$,
which finishes the proof of Proposition 3.1.
\end{proof}
Proposition 3.1 is true for any sub-square $Q\subset(0,1)^{2}$ of
side length $1/2$, because Green functions are translation invariant
(i.e., $G_{Q+z}\left(u+z,v+z\right)=G_{Q}(u,v)$ for any $z\in\mathbb{R}^{2},u,v\in Q$,
where $G_{Q+z}$ is the Green function of $Q+z$ for the operator
$-\Delta$, with Dirichlet boundary conditions on $\partial Q+z$).

\subsubsection*{Estimates on the covariance}

In this subsection, we prove that the ``bulk'' of the field $\left\{ \sqrt{\frac{\pi}{2}}X_{\epsilon}^{x}:x\in[0,1]^{2}\right\} $
satisfies both \eqref{logcorrY} and \eqref{l2distY}. Recall that
$\Gamma(\cdot,\cdot)=\Gamma\left(\left\Vert \cdot-\cdot\right\Vert \right)=\frac{2}{\pi}\log(1/\left\Vert \cdot-\cdot\right\Vert )$
is the Green function of $\mathbb{R}^{2}$ for the operator $-\Delta$.
\begin{prop}
\label{prop3.4} Let $K\subset(0,1)^{2}$ be such that $k=dist(\partial(0,1)^{2},K)>0$,
and let $0<\epsilon<k/2$. Then, there exists a constant $C<\infty$,
depending on $k$ only, such that, for all $x\in K,y\in[0,1]^{2}$,

\[
\left|\frac{1}{\pi\epsilon^{2}}\int_{D(x,\epsilon)}G(u,y)du-\frac{2}{\pi}\log(1/\epsilon)\right|\leq C
\]
if $\left\Vert y-x\right\Vert <\epsilon$, and 
\[
\left|\frac{1}{\pi\epsilon^{2}}\int_{D(x,\epsilon)}G(u,y)du-\frac{2}{\pi}\log(1/\left\Vert x-y\right\Vert )\right|\leq C
\]
if $\left\Vert y-x\right\Vert \geq\epsilon$. \end{prop}
\begin{proof}
The function $(G-\Gamma)(x,y)$ is symmetric, harmonic in each variable,
and continuous. Hence, 

\[
\left|\frac{1}{\pi\epsilon^{2}}\int_{D(x,\epsilon)}(G-\Gamma)(u,y)\, du\right|\leq\sup_{\{u:dist(u,\partial(0,1)^{2})\geq k/2\}}\sup_{\{y\in[0,1]^{2}\}}\left|(G-\Gamma)(u,y)\right|
\]

\[
\leq\sup_{u:dist(u,\partial(0,1)^{2})\geq k/2}|\Gamma(dist(u,\partial(0,1)^{2}))|=\Gamma(k/2),
\]
where the second bound is obtained by applying the maximum principle
to $(G-\Gamma)(u,\cdot)$, noting that $G(u,\cdot)$ vanishes at the
boundary of $(0,1)^{2}$, and using that $\Gamma$ is decreasing.
Therefore, it is enough to prove Proposition \ref{prop3.4} with $G$
replaced by $\Gamma$.

Suppose that $\left\Vert x-y\right\Vert <\epsilon$. Then, 
\[
\left|\frac{1}{\pi\epsilon^{2}}\int_{D(x,\epsilon)}\Gamma(u,y)-\Gamma(\epsilon)du\right|=\left|\frac{1}{\pi\epsilon^{2}}\int_{D(x,\epsilon)}\Gamma\left(\frac{\left\Vert u-y\right\Vert }{\epsilon}\right)du\right|.
\]
The change of variables $u'=(u-y)/\epsilon$ implies that the previous
display is 
\[
=\left|\frac{1}{\pi}\int_{D((x-y)/\epsilon,1)}\Gamma(u')du'\right|\leq\sup_{z\in\overline{D(0,1)}}\left|\frac{1}{\pi}\int_{D(z,1)}\Gamma(u')du'\right|\leq C,
\]
by continuity in $z$ and compactness of $\overline{D(0,1)}$, where
$C$ is an absolute constant.

Suppose now that $\left\Vert x-y\right\Vert \geq\epsilon$. Then,
\[
\left|\frac{1}{\pi\epsilon^{2}}\int_{D(x,\epsilon)}\Gamma(u,y)du-\Gamma(\left\Vert x-y\right\Vert )\, du\right|=\left|\frac{1}{\pi\epsilon^{2}}\int_{D(x,\epsilon)}\Gamma\left(\left\Vert \frac{u-y}{\left\Vert x-y\right\Vert }\right\Vert \right)du\right|.
\]
The change of variables $u'=(u-y)/\left\Vert x-y\right\Vert $ implies
that the previous line is 
\[
=\left|\frac{1}{\pi(\epsilon/\left\Vert x-y\right\Vert )^{2}}\int_{D\left(\frac{x-y}{\left\Vert x-y\right\Vert },\frac{\epsilon}{\left\Vert x-y\right\Vert }\right)}\Gamma(u')\, du'\right|\leq\sup_{0\leq r\leq1}\sup_{\left\Vert z\right\Vert =1}\left|\frac{1}{\pi r^{2}}\int_{D(z,r)}\Gamma(u')\, du'\right|<C,
\]
by continuity in $r,z$ and compactness of $\left\{ 0\leq r\leq1\right\} \times\left\{ \left\Vert z\right\Vert =1\right\} $,
where $C$ is an absolute constant. 
\end{proof}
Note that the fact that we are integrating over disks is not essential.
We could define similar MGFF for other mollifiers.

A trivial corollary (which follows from elementary properties of $\log$)
of the previous proposition is
\begin{cor}
\label{cor3.5} Let $K,k,\epsilon$ be as in Proposition \ref{prop3.4}
and let $c_{0}>0$. Then, there exists a constant $C$ (depending
on $k$ and $c_{0}$) such that, for all $x\in K,y\in[0,1]^{2}$,

\[
\left|\frac{1}{\pi\epsilon^{2}}\int_{D(x,\epsilon)}G(u,y)du-\frac{2}{\pi}\log(1/\epsilon)\right|\leq C
\]
whenever $\left\Vert x-y\right\Vert <c_{0}\epsilon$, and

\[
\left|\frac{1}{\pi\epsilon^{2}}\int_{D(x,\epsilon)}G(u,y)du-\frac{2}{\pi}\log(1/\left\Vert x-y\right\Vert )\right|\leq C
\]
whenever$\left\Vert x-y\right\Vert \geq c_{0}\epsilon$.
\end{cor}
Now we prove an important corollary of Proposition \ref{prop3.4}. 
\begin{cor}
\label{cor3.6} Let $K,k$ be as in Proposition \ref{prop3.4}. Then,
there exists a constant $C$ (depending only on $k$) such that, for
all $x,y\in K,\epsilon>0$, 
\begin{equation}
|Cov(X_{\epsilon}^{x},X_{\epsilon}^{y})+\frac{2}{\pi}\log(\max\{\epsilon,\left\Vert x-y\right\Vert \})|\leq C.\label{cov}
\end{equation}
Moreover, if $\left\Vert x-y\right\Vert \leq\epsilon$, then 
\begin{equation}
\mathbb{E}\left(X_{\epsilon}^{x}-X_{\epsilon}^{y}\right)^{2}\leq C\epsilon^{-1}\left\Vert x-y\right\Vert .\label{l2_1}
\end{equation}
 \end{cor}
\begin{proof}
Let us prove \eqref{cov}. If $\left\Vert x-y\right\Vert \leq2\epsilon$,
by Corollary \ref{cor3.5}, 
\[
\left|\frac{1}{\pi\epsilon^{2}}\int_{D(y,\epsilon)}G(u,v)\, dv-\Gamma(\epsilon)\right|\leq C
\]
for every $u\in D(x,\epsilon)$. Integrating the last inequality over
$u\in D(x,\epsilon)$ and using \eqref{green}, we obtain that

\[
\left|Cov(X_{\epsilon}^{x},X_{\epsilon}^{y})-\Gamma(\epsilon)\right|\leq C
\]
for all $\left\Vert x-y\right\Vert \leq2\epsilon$ (and in particular,
for $\left\Vert x-y\right\Vert \leq\epsilon$).

If $\left\Vert x-y\right\Vert \geq2\epsilon$, Corollary \ref{cor3.5}
implies 
\[
\left|\frac{1}{\pi\epsilon^{2}}\int_{D(y,\epsilon)}G(u,v)dv-\Gamma(\left\Vert y-u\right\Vert )\right|\leq C
\]
for every $u\in D(x,\epsilon)$. But $\Gamma(3/2)\leq\Gamma(1+\tfrac{\epsilon}{\left\Vert x-y\right\Vert })\leq\Gamma(\left\Vert y-u\right\Vert )-\Gamma(\left\Vert x-y\right\Vert )\leq\Gamma(1-\tfrac{\epsilon}{\left\Vert x-y\right\Vert })\leq\Gamma(1/2)$
for all $u\in D(x,\epsilon)$. Therefore,

\[
\left|\frac{1}{\pi\epsilon^{2}}\int_{D(y,\epsilon)}G(u,v)dv-\Gamma(\left\Vert x-y\right\Vert )\right|\leq C.
\]
The same (with a different constant) holds for $\left\Vert x-y\right\Vert \geq\epsilon$,
because $\Gamma$ is logarithmic. Integrating over $u\in D(x,\epsilon)$
finishes the proof of \eqref{cov}.

We now prove \eqref{l2_1}. Display \eqref{green} implies 
\[
Cov(X_{\epsilon}^{x},X_{\epsilon}^{x}-X_{\epsilon}^{y})=\frac{1}{\pi^{2}\epsilon^{4}}\int_{D(x,\epsilon)}\int_{D(x,\epsilon)}G-\frac{1}{\pi^{2}\epsilon^{4}}\int_{D(x,\epsilon)}\int_{D(y,\epsilon)}G
\]
 
\[
=\frac{1}{\pi^{2}\epsilon^{4}}\left(\int_{D(x,\epsilon)\backslash D(y,\epsilon)}\int_{D(x,\epsilon)}G-\int_{D(y,\epsilon)\backslash D(x,\epsilon)}\int_{D(x,\epsilon)}G\right).
\]
We can use Corollary \ref{cor3.5} to obtain an upper bound of the
first term and a lower bound of the second term of the previous display.
Then, the previous display is 
\[
\leq\frac{1}{\pi\epsilon^{2}}\left(\int_{D(x,\epsilon)\backslash D(y,\epsilon)}(\Gamma(\epsilon)+C)-\int_{D(y,\epsilon)\backslash D(x,\epsilon)}(\Gamma(\epsilon)-C)\right)=\frac{C}{\pi\epsilon^{2}}\left|D(x,\epsilon)\backslash D(y,\epsilon)\right|,
\]
where $\left|D(x,\epsilon)\backslash D(y,\epsilon)\right|$ is the
Lebesgue measure of the set $D(x,\epsilon)\backslash D(y,\epsilon)$.
Elementary geometry implies $|D(x,\epsilon)\backslash D(y,\epsilon)|\leq C\epsilon\left\Vert x-y\right\Vert $.
Repeating the previous argument for $Cov(X_{\epsilon}^{y},X_{\epsilon}^{y}-X_{\epsilon}^{x})$
finishes the proof.
\end{proof}

\subsection{Tightness for the MGFF}

In the next theorem we provide upper bounds on the left and right
tail of the MGFF, and we compute the expected maximum up to an order
1 term.
\begin{thm}
For $\epsilon>0$, let $X_{\epsilon}^{x},x\in[0,1]^{2}$ be the MGFF.
Then, there exist absolute constants $0<c,C<\infty$ such that

\begin{equation}
\mathbb{P}\left(\left|\max_{x\in[0,1]^{2}}X_{\epsilon}^{x}-\sqrt{\frac{2}{\pi}}m_{\epsilon}\right|\geq+\lambda\right)\leq Ce^{-c\lambda}\label{rtail}
\end{equation}
for all $\lambda\geq0$. Moreover,

\[
\mathbb{E}\left[\max_{x\in[0,1]^{2}}X_{\epsilon}^{x}\right]=\sqrt{\frac{2}{\pi}}m_{\epsilon}+O(1).
\]
\end{thm}
\begin{proof}
Let $Q$ be the open square of side length $1/2$, which is concentric
with $(0,1)^{2}$, and let $q:[0,1]^{2}\to\overline{Q}$ be the natural
concentric contraction. Consider the field $Y_{\epsilon}^{x}:=X_{\epsilon/2}^{q(x)};x\in[0,1]^{2}$.
By Corollary \ref{cor3.6},

\[
Cov(Y_{\epsilon}^{x},Y_{\epsilon}^{y})=Cov(X_{\epsilon/2}^{q(x)},X_{\epsilon/2}^{q(y)})=\frac{2}{\pi}\log\left(\max\left\{ \epsilon/2,\left\Vert q(x)-q(y)\right\Vert \right\} \right)+O(1)
\]

\[
=\frac{2}{\pi}\log\left(\max\left\{ \epsilon,\left\Vert x-y\right\Vert \right\} \right)+O(1)
\]
for all $x,y\in[0,1]^{2}$, and

\[
\mathbb{E}\left(Y_{\epsilon}^{x}-Y_{\epsilon}^{y}\right)^{2}=\mathbb{E}\left(X_{\epsilon/2}^{q(x)}-X_{\epsilon/2}^{q(y)}\right)^{2}\leq C\epsilon^{-1}2\left\Vert q(x)-q(y)\right\Vert =C\epsilon^{-1}\left\Vert x-y\right\Vert 
\]
for all $x,y\in[0,1]^{2}$ such that $\left\Vert x-y\right\Vert \leq\epsilon$.
An application of Theorem \ref{main} yields the existence of absolute
constants $0<c,C<\infty$ such that

\begin{equation}
\mathbb{P}\left(\max_{x\in[0,1]^{2}}Y_{\epsilon}^{x}-\sqrt{\frac{2}{\pi}}m_{\epsilon}\geq\lambda\right)=\mathbb{P}\left(\max_{x\in\overline{Q}}X_{\epsilon/2}^{x}-\sqrt{\frac{2}{\pi}}m_{\epsilon}\geq\lambda\right)\leq Ce^{-c\lambda}\label{asdtail}
\end{equation}
and 

\begin{equation}
\mathbb{P}\left(\max_{x\in\overline{Q}}X_{\epsilon/2}^{x}-\sqrt{\frac{2}{\pi}}m_{\epsilon}\leq\lambda\right)\leq Ce^{-c\lambda}\label{asdasdtail}
\end{equation}
for all $\lambda\geq0$. Bound \eqref{asdasdtail} easily implies
that

\[
\mathbb{P}\left(\max_{x\in[0,1]^{2}}X_{\epsilon/2}^{x}-\sqrt{\frac{2}{\pi}}m_{\epsilon}\leq\lambda\right)\leq\mathbb{P}\left(\max_{x\in\overline{Q}}X_{\epsilon/2}^{x}-\sqrt{\frac{2}{\pi}}m_{\epsilon}\leq\lambda\right)\leq Ce^{-c\lambda}
\]
for all $\lambda\geq0$, proving \eqref{rtail} for the left tail
(after using $m_{\epsilon/2}=m_{\epsilon}+O(1)$, and adjusting the
constants).

In order to prove the bound \eqref{rtail} for the right tail, we
use Proposition \ref{prop3.1} and the comment that follows it to
decompose

\[
X_{\epsilon/2}^{x}=\hat{X}_{\epsilon/2}^{x}+\phi^{x},
\]
where $\left(\hat{X}_{\epsilon/2}^{x}:x\in\overline{Q}\right)\overset{d}{=}\left(X_{\epsilon}^{x}:x\in[0,1]^{2}\right)$
and the fields $\left(\phi^{x}:x\in\overline{Q}\right),\left(\hat{X}_{\epsilon/2}^{x}:x\in\overline{Q}\right)$
are independent. If $\chi=\arg\max\left\{ \hat{X}_{\epsilon/2}^{x}:x\in\overline{Q}\right\} $,
then

\[
\left\{ \phi^{\chi}\geq0,\hat{X}_{\epsilon/2}^{\chi}-\sqrt{\frac{2}{\pi}}m_{\epsilon}\geq\lambda\right\} \subset\left\{ \max_{x\in\overline{Q}}X_{\epsilon/2}^{x}-\sqrt{\frac{2}{\pi}}m_{\epsilon}\geq\lambda\right\} .
\]
But independence of $\phi$ and $\chi$ implies

\[
\mathbb{P}\left(\phi^{\chi}\geq0,\hat{X}_{\epsilon/2}^{\chi}-\sqrt{\frac{2}{\pi}}m_{\epsilon}\geq\lambda\right)=\frac{1}{2}\mathbb{P}\left(\hat{X}_{\epsilon/2}^{\chi}-\sqrt{\frac{2}{\pi}}m_{\epsilon}\geq\lambda\right)
\]
because $\phi$ is a centered field. By using the last display and
\eqref{asdtail}, we obtain

\[
\mathbb{P}\left(\hat{X}_{\epsilon/2}^{\chi}-\sqrt{\frac{2}{\pi}}m_{\epsilon}\geq\lambda\right)\leq2\mathbb{P}\left(\max_{x\in\overline{Q}}X_{\epsilon/2}^{x}-\sqrt{\frac{2}{\pi}}m_{\epsilon}\geq\lambda\right)\leq Ce^{-c\lambda}
\]
for some absolute constants $0<c,C<\infty$.

The bound \eqref{rtail} and $m_{\epsilon/2}=m_{\epsilon}+O(1)$ implies
tightness of the family

\[
\left\{ \max_{x\in[0,1]^{2}}X_{\epsilon}^{x}-m_{\epsilon}:\epsilon>0\right\} ,
\]
and the same bound also implies 

\[
\mathbb{E}\left[\max_{x\in[0,1]^{2}}X_{\epsilon}^{x}\right]=\sqrt{\frac{2}{\pi}}m_{\epsilon}+O(1),
\]
finishing the proof.
\end{proof}

\section{Appendix}

We prove here the claims made in Section 2.1.
\begin{prop}
\label{proofdefMBRW}The MBRW, defined by display \eqref{defMBRW},
exists and satisfies
\[
Var(\xi_{\epsilon}^{v}(t))=t
\]
for all $0\leq t\leq\log(1/\epsilon)$ and all $v\in V_{\epsilon}$,
and
\[
t-C\leq Cov(\xi_{\epsilon}^{v}(t),\xi_{\epsilon}^{w}(t))\leq t
\]
for all $0\leq t\leq-\log\left\Vert v-w\right\Vert _{\infty}$ and
all $v,w\in V_{\epsilon}$, where $C$ is a constant depending on
the dimension. \end{prop}
\begin{proof}
We show that the mapping $\left(V_{\epsilon}\times[0,\log(1/\epsilon)]\right)^{2}\to\mathbb{R}$
given by 
\[
\left((v,t),(u,s)\right)\mapsto\int_{0}^{\min\left\{ t,s\right\} }\prod_{i\leq d}(1-e^{r}\left|v_{i}-u_{i}\right|)_{+}dr
\]
is positive definite. Note first that
\[
\prod_{i\leq d}(1-e^{r}\left|v_{i}-u_{i}\right|)_{+}=\int_{\mathbb{R}^{d}}1_{A(v,r)}(z)1_{A(u,r)}(z)dz,
\]
where $dz$ is $d$-dimensional Lebesgue measure and $A(v,r)$ is
the $d$-dimensional box of side length 1, centered at $e^{r}v$.
Let $\left\{ (v^{\alpha},t^{\alpha})\right\} _{\alpha}$ be any finite
subset of $V_{\epsilon}\times[0,\log(1/\epsilon)]$, and let $\left\{ c_{\alpha}\right\} _{\alpha}$
be arbitrary real numbers. Then, applying the previous display, we
obtain
\[
\sum_{\alpha,\beta}c_{\alpha}c_{\beta}\int_{0}^{\min\left\{ t^{\alpha},t^{\beta}\right\} }\prod_{i\leq d}(1-e^{r}\left|v_{i}^{\alpha}-v_{i}^{\beta}\right|)_{+}dr
\]
\[
=\int_{0}^{\infty}\int_{\mathbb{R}^{d}}\sum_{\alpha,\beta}c_{\alpha}c_{\beta}1_{[0,t^{\alpha}]}(r)1_{[0,t^{\beta}]}(r)1_{A(v^{\alpha},r)}(z)1_{A(v^{\beta},r)}(z)\, dz\, dr
\]
\[
=\int_{0}^{\infty}\int_{\mathbb{R}^{d}}\left(\sum_{\alpha}c_{\alpha}1_{[0,t^{\alpha}]}(r)1_{A(v^{\alpha},r)}(z)\right)^{2}\, dz\, dr\geq0,
\]
as desired. This shows that the MBRW exists. 

For any $v\in V_{\epsilon}$ and $t\leq\log(1/\epsilon)$,
\[
Var(\xi_{\epsilon}^{v}(t))=\int_{0}^{t}\prod_{i\leq d}\left(1\right)dr=t.
\]
Moreover, if $v\neq w$,
\[
\prod_{i\leq d}\left(1-e^{r}\left|v_{i}-w_{i}\right|\right)_{+}\begin{cases}
>0 & \text{ if }r<-\log\left\Vert v-w\right\Vert _{\infty}\\
=0 & \text{ if }r\geq-\log\left\Vert v-w\right\Vert _{\infty}
\end{cases}.
\]
Therefore, if $t<-\log\left\Vert v-w\right\Vert _{\infty}$,
\[
t\geq Cov(\xi_{\epsilon}^{v}(t),\xi_{\epsilon}^{w}(t))\geq\int_{0}^{t}\prod_{i\leq d}\left(1-e^{r}\left|v_{i}-w_{i}\right|\right)dr\geq\int_{0}^{t}\left(1-e^{r}\left\Vert v-w\right\Vert _{\infty}\right)^{d}dr
\]
\[
\geq t+\sum_{k=1}^{d}\left(\begin{array}{c}
d\\
k
\end{array}\right)(-1)^{k}\left\Vert v-w\right\Vert _{\infty}^{k}\left(\frac{e^{kt}-1}{k}\right)\geq t-C
\]
for some constant $C<\infty$ depending on $d$ only. Similarly, if
$t\geq-\log\left\Vert v-w\right\Vert _{\infty}$,
\[
-\log\left\Vert v-w\right\Vert _{\infty}\geq Cov(\xi_{\epsilon}^{v}(t),\xi_{\epsilon}^{w}(t))\geq-\log\left\Vert v-w\right\Vert _{\infty}-C.
\]
\end{proof}
\begin{prop}
\label{proofrtlbMBRW} Let $\left(\xi_{\epsilon}^{v}:v\in V_{\epsilon}\right)$
be the MBRW and let $m_{\epsilon}$ be the number defined in the line
preceding Theorem \ref{main}. Then, there exists a constant $c>0$
(depending on the dimension) such that
\[
\mathbb{P}\left(\max_{v\in V_{\epsilon}}\xi_{\epsilon}^{v}\geq m_{\epsilon}\right)\geq c.
\]
\end{prop}
\begin{proof}
We use a second moment method. Let $T=T_{\epsilon}=\log(1/\epsilon)$
and let 
\[
A_{v}=\left\{ \xi_{\epsilon}^{v}\geq m_{\epsilon},\xi_{\epsilon}^{v}(t)\leq\frac{m_{\epsilon}}{T}t+1\text{ for all }0\leq t\leq T\right\} ,
\]
\[
Z=\sum_{v\in V_{\epsilon}}1_{A_{v}}.
\]
Note that
\begin{equation}
\mathbb{P}\left(\max_{v\in V_{\epsilon}}\xi_{\epsilon}^{v}\geq m_{\epsilon}\right)\geq\mathbb{P}\left(Z>0\right)\geq\frac{\left(\mathbb{E}\left[Z\right]\right)^{2}}{\mathbb{\mathbb{E}}\left[Z^{2}\right]},\label{secondmomentmethod}
\end{equation}
where the second inequality follows by Cauchy-Schwarz. We first compute
a lower bound for $\mathbb{E}\left[Z\right]$. Note that
\[
\mathbb{E}\left[Z\right]=\epsilon^{-d}\mathbb{P}\left(A_{v}\right).
\]
Let $\bar{\xi}_{\epsilon}^{v}(t)=\xi_{\epsilon}^{v}(t)-\frac{m_{\epsilon}}{T}t$.
Define a probability measure $\mathbb{Q}$ by
\[
\frac{d\mathbb{P}}{d\mathbb{Q}}=\exp\left(-\frac{m_{\epsilon}}{T}\bar{\xi}_{\epsilon}^{v}(T)-\frac{m_{\epsilon}^{2}}{2T}\right).
\]
Girsanov's Theorem (see \cite[Theorem 5.1]{ks}) implies that $\bar{\xi}_{\epsilon}^{v}(t)$
is Brownian motion under $\mathbb{Q}$. Note that
\[
\mathbb{P}\left(A_{v}\right)=\int_{A_{v}}\exp\left(-\frac{m_{\epsilon}}{T}\bar{\xi}_{\epsilon}^{v}(T)-\frac{m_{\epsilon}^{2}}{2T}\right)d\mathbb{Q}\geq\exp\left(-\frac{m_{\epsilon}}{T}-\frac{m_{\epsilon}^{2}}{2T}\right)\mathbb{Q}(A_{v})
\]
\[
\geq ce^{-\sqrt{2d}}\epsilon^{d}T^{3/2}\mathbb{Q}(A_{v})
\]
for some absolute constant $c>0$. It follows easily from the Reflection
Principle (see \cite[Proposition 6.19]{ks}) that $\mathbb{Q}(A_{v})=\mathbb{Q}(\bar{\xi}_{\epsilon}^{v}\geq0,\bar{\xi}_{\epsilon}^{v}(t)\leq1\text{ for all }0\leq t\leq T)\geq cT^{-3/2}$
for some absolute constant $c>0$. Combining the three previous displays,
we obtain
\begin{equation}
\mathbb{E}\left[Z\right]\geq c\label{lbfirstmoment}
\end{equation}
for some constant $c>0$, depending on the dimension $d$. 

We now compute an upper bound for $\mathbb{E}\left[Z^{2}\right]$.
Note that
\begin{equation}
\mathbb{E}\left[Z^{2}\right]=\sum_{v,w\in V_{\epsilon}}\mathbb{P}\left(A_{v}\cap A_{w}\right)=\sum_{v,w\in V_{\epsilon}}\mathbb{P}\left(\bar{\xi}_{\epsilon}^{v},\bar{\xi}_{\epsilon}^{w}\geq0,\bar{\xi}_{\epsilon}^{v}(t),\bar{\xi}_{\epsilon}^{w}(t)\leq1\text{ for all }0\leq t\leq T\right).\label{ZZ}
\end{equation}
Both $\xi_{\epsilon}^{v}(\cdot),\xi_{\epsilon}^{w}(\cdot)$ are Brownian
motions, which have independent increments starting at time $s=s_{v,w}=-\log\left(\max\left\{ \epsilon,\left\Vert v-w\right\Vert _{\infty}\right\} \right)$.
Therefore,
\[
\mathbb{P}\left(A_{v}\cap A_{w}\right)\leq\sum_{-\infty<x,y\leq1}p(x)p(y)\mathbb{P}\left(\bar{\xi}_{\epsilon}^{v}(t),\bar{\xi}_{\epsilon}^{w}(t)\leq1\text{ for all }t\in[0,s],\bar{\xi}_{\epsilon}^{v}(s)\in[x-1,x],\bar{\xi}_{\epsilon}^{w}(s)\in[y-1,y]\right)
\]
\begin{equation}
\leq\sum_{-\infty<y\leq x\leq1}2p(x)p(y)\mathbb{P}\left(\bar{\xi}_{\epsilon}^{v}(t),\bar{\xi}_{\epsilon}^{w}(t)\leq1\text{ for all }t\in[0,s],\bar{\xi}_{\epsilon}^{v}(s)\in[x-1,x],\bar{\xi}_{\epsilon}^{w}(s)\in[y-1,y]\right),\label{conditioning}
\end{equation}
where
\[
p(x)=\sup_{z\in[x-1,x]}\mathbb{P}\left(\bar{\xi}_{\epsilon}^{v}(t)\leq1-z\text{ for all }t\in[0,T-s],\bar{\xi}_{\epsilon}^{v}(T-s)\geq-z\right).
\]

Assume $0<s<T$. Applying Girsanov's Theorem and the Reflection Principle,
we obtain 
\[
p(x)\leq C\exp\left(\frac{m_{\epsilon}}{T}x-\frac{m_{\epsilon}^{2}}{2T^{2}}(T-s)\right)\frac{(1-x)}{\left(T-s\right)^{3/2}}
\]
for some constant $C$. Therefore, from \eqref{conditioning} and
the last display,
\[
\mathbb{P}\left(A_{v}\cap A_{w}\right)\leq\sum_{-\infty<y\leq x\leq1}Cp(x)^{2}\mathbb{P}\left(\bar{\xi}_{\epsilon}^{v}(t),\bar{\xi}_{\epsilon}^{w}(t)\leq1\text{ for all }t\in[0,s],\bar{\xi}_{\epsilon}^{v}(s)\in[x-1,x],\bar{\xi}_{\epsilon}^{w}(s)\in[y-1,y]\right)
\]
\[
\leq\sum_{-\infty<x\leq1}Cp(x)^{2}\mathbb{P}\left(\bar{\xi}_{\epsilon}^{v}(t)\leq1\text{ for all }t\in[0,s],\bar{\xi}_{\epsilon}^{v}(s)\in[x-1,x]\right).
\]
Applying Girsanov's Theorem and the Reflection Principle again,
\[
\mathbb{P}\left(A_{v}\cap A_{w}\right)\leq C\sum_{-\infty<x\leq1}p(x)^{2}\exp\left(-\frac{m_{\epsilon}}{T}x-\frac{m_{\epsilon}^{2}}{2T^{2}}s\right)\frac{(1-x)}{s^{3/2}}
\]
\begin{equation}
\leq C\frac{1}{\left(T-s\right)^{3}s^{3/2}}\exp\left(-\frac{m_{\epsilon}^{2}}{2T^{2}}(2T-s)\right)\label{intersection1}
\end{equation}
for some constant $C$. 

Consider now the case $s=0$. Then, the independence of $\xi_{\epsilon}^{v}(\cdot)$
and $\xi_{\epsilon}^{v}(\cdot)$ implies 
\[
\mathbb{P}\left(A_{v}\cap A_{w}\right)=\mathbb{P}\left(A_{v}\right)^{2}=\mathbb{P}\left(\bar{\xi}_{\epsilon}^{v}(t)\leq1\text{ for all }t\in[0,T],\bar{\xi}_{\epsilon}^{v}(T)\geq0\right)^{2}
\]
\begin{equation}
\leq C\frac{1}{T^{3}}\exp\left(-\frac{m_{\epsilon}^{2}}{T}\right),\label{intersection2}
\end{equation}
where the last bound follows from Girsanov's Theorem and the Reflection
Principle. In the case $s=T$,
\begin{equation}
\mathbb{P}\left(A_{v}\cap A_{w}\right)\leq\mathbb{P}\left(A_{v}\right)\leq C\frac{1}{T^{3/2}}\exp\left(-\frac{m_{\epsilon}^{2}}{2T}\right).\label{intersection3}
\end{equation}
In consequence, for any pair $v,w\in V_{\epsilon}$, displays \eqref{intersection1},
\eqref{intersection2} and \eqref{intersection3} imply 
\[
\mathbb{P}\left(A_{v}\cap A_{w}\right)\leq C\frac{1}{\left(\left(T-s\right)\vee1\right)^{3}\left(s\vee1\right)^{3/2}}\exp\left(-\frac{m_{\epsilon}^{2}}{2T^{2}}(2T-s)\right),
\]
where $\cdot\vee\cdot=\max\left\{ \cdot,\cdot\right\} $. For any
fixed $v\in V_{\epsilon}$, there are $O(e^{(d-1)(T-s)})$ points
$w$ such that $-\log\left\Vert v-w\right\Vert _{\infty}=s$. Therefore,
from \eqref{ZZ}, the last display, we obtain
\[
\mathbb{E}\left[Z^{2}\right]\leq C\sum_{0\leq s\leq T}\left|V_{\epsilon}\right|e^{(d-1)(T-s)}\frac{1}{\left(\left(T-s\right)\vee1\right)^{3}\left(s\vee1\right)^{3/2}}\exp\left(-\frac{m_{\epsilon}^{2}}{2T^{2}}(2T-s)\right)
\]
\[
\leq C+C\sum_{0<s<T}\left|V_{\epsilon}\right|e^{(d-1)(T-s)}\frac{\exp\left(-\frac{m_{\epsilon}^{2}}{2T^{2}}(2T-s)\right)}{\left(T-s\right)^{3}s^{3/2}}
\]
\[
=C+C\sum_{0<s<T}e^{dT}e^{(d-1)(T-s)}\frac{\exp\left(-\frac{m_{\epsilon}^{2}}{2T^{2}}(2T-s)\right)}{\left(T-s\right)^{3}s^{3/2}}.
\]
But,
\[
\sum_{0<s<T}e^{dT}e^{(d-1)(T-s)}\frac{\exp\left(-\frac{m_{\epsilon}^{2}}{2T^{2}}(2T-s)\right)}{\left(T-s\right)^{3}s^{3/2}}\leq\sum_{0<s<T}e^{d(2T-s)}\frac{\exp\left(\left(-d+\frac{3\log T}{2T}\right)(2T-s)\right)}{\left(T-s\right)^{3}s^{3/2}}
\]
\[
=\sum_{0<s<T}\frac{\exp\left(\frac{3}{2}\frac{\log T}{T}(2T-s)\right)}{(T-s)^{3}s^{3/2}}\leq C\sum_{0<s<T/2}\frac{1}{s^{3/2}}+\sum_{T/2\leq s<T}\frac{\exp\left(\frac{3}{2}\frac{\log T}{T}(T-s)\right)}{(T-s)^{3}}\frac{T^{3/2}}{s^{3/2}}
\]
\[
\leq C+C\sum_{0<s\leq T/2}\frac{T^{3s/2T}}{s^{3}}\leq C<\infty,
\]
because the last expression is (eventually) decreasing in $T$. Proposition
\ref{proofrtlbMBRW} follows from the last display, \eqref{secondmomentmethod}
and \eqref{lbfirstmoment}. \end{proof}
\begin{prop}
\label{proofrtubMBRW} Let $\left(\xi_{\epsilon}^{v}:v\in V_{\epsilon}\right)$
be the MBRW and let $m_{\epsilon}$ be the number defined in the line
preceding Theorem \ref{main}. Then, there exist constants $0<c,C<\infty$
(depending on the dimension $d$) such that
\[
\mathbb{P}\left(\max_{v\in A}\xi_{\epsilon}^{v}\geq m_{\epsilon}+z\right)\leq C\left(\epsilon^{d}\left|A\right|\right)^{1/2}e^{-cz}
\]
for all $A\subset V_{\epsilon}$, $z\in\mathbb{R}$ and $\epsilon>0$
small enough.\end{prop}
\begin{proof}
We introduce the $d$-ary branching random walk (BRW) as follows:
let $\epsilon=2^{-n}$ for some $n\in\mathbb{N}$. At each time $T_{k}=k\log2;k=0,1,\ldots,n$,
we partition $[0,1)^{d}$ into $2^{kd}$ disjoint boxes of side length
$2^{-k}$. For a pair $v,w\in V_{\epsilon}$, denote by $l(v,w)$
the first time that $v,w$ lie in different boxes of the partition.
With this notation, define the BRW as the Gaussian field $\left(\eta_{\epsilon}^{v}(t):v\in V_{\epsilon},t\in[0,T_{n}]\right)$
with
\[
Cov(\eta_{\epsilon}^{v}(t),\eta_{\epsilon}^{w}(s))=\min\left\{ t,s,l(v,w)\right\} .
\]
For simplicity, let $T=T_{n}$ and $\eta_{\epsilon}^{v}=\eta_{\epsilon}^{v}(T)$.
It is not hard to show that such a field exists. Note that our BRW
can be interpreted as a branching Brownian motion that splits every
$\log2$ units of time into $2^{d}$ independent Brownian motions.
Following the argument given in \cite[Lemma 3.7]{dz}, one can show
that there exists $C$ (depending on the dimension) such that
\[
\mathbb{P}\left(\max_{v\in A}\xi_{\epsilon}^{v}\geq m_{\epsilon}+\lambda\right)\leq C\mathbb{P}\left(\max_{v\in A}\eta_{\epsilon/C}^{v}\geq m_{\epsilon}+\lambda\right)
\]
for all $A\subset V_{\epsilon}\subset V_{\epsilon/C}$ and all $\lambda\in\mathbb{R}$.
Therefore, it is enough to prove Proposition \ref{proofrtubMBRW}
for the BRW. We do so by following very closely the proof in \cite[Lemma 3.8]{bdz}. 

We will use the following estimate, which is proved in \cite[Lemma 3.6]{bdz}:
let $W_{s}$ be standard Brownian motion under $\mathbb{P}$ and fix
a large constant $C_{1}$. Then, if 
\[
\mu_{q,r}^{\ast}(x)=\mathbb{P}\left(W_{q}\in dx,W_{s}\leq r+C_{1}(\min\left\{ s,q-s\right\} )^{1/20}\text{ for all }0\leq s\leq q\right)/dx,
\]
we have
\begin{equation}
\mu_{q,r}^{\ast}(x)\leq C_{2}r(r-x)/q^{3/2}\label{quasibridge}
\end{equation}
for all $x\leq r$, where $C_{2}$ depends on $C_{1}$.

We next define the event
\[
G(\lambda)=\left\{ \exists t\leq T,v\in V_{\epsilon}:\eta_{\epsilon}^{v}(t)-\frac{m_{\epsilon}}{T}t-10\log\left(\min\left\{ t,T-t\right\} \right)_{+}\geq\lambda\right\} 
\]
and we prove the following claim:
\begin{claim}
There exists a constant $C>0$ (depending on $d$) such that
\[
\mathbb{P}\left(G(\lambda)\right)\leq C\lambda e^{-\sqrt{2d}\lambda}
\]
for all $\lambda\geq1$. \end{claim}
\begin{proof}
Following the proof of \cite[Lemma 3.7]{bdz}, we define $\psi_{t}=\lambda+10\log\left(\min\left\{ t,T-t\right\} \right)_{+}$
and $\chi_{T_{k}}(x)=\mathbb{P}\left(\eta_{\epsilon}^{v}(t)-\frac{m_{\epsilon}}{T}t\leq\psi_{t}\text{ for all }t\leq T_{k},\eta_{\epsilon}^{v}(T_{k})-\frac{m_{\epsilon}}{T}T_{k}\in dx\right)/dx$.
Then, by decomposing based on the first time such that $\eta_{\epsilon}^{v}(t)-\frac{m_{\epsilon}}{T}t\geq\psi_{t}$,
we obtain that 
\[
\mathbb{P}\left(G(\lambda)\right)\leq\sum_{k=1}^{n}2^{dk}\int_{-\infty}^{\psi_{T_{k}}}\chi_{T_{k}}(x)\mathbb{P}\left(\max_{s\leq\log2}\eta_{\epsilon}^{v}(t)\geq\psi_{T_{k}}-x-C\right)dx,
\]
where $C$ is an absolute constant. Display \eqref{quasibridge} and
Girsanov's Theorem imply that
\[
\chi_{T_{k}}(x)\leq C2^{-dk}e^{-x\left(\sqrt{2d}-O(\log T/T)\right)}\psi_{T_{k}}(\psi_{T_{k}}-x),
\]
where $C$ depends on $d$. On the other hand,
\[
\mathbb{P}\left(\max_{s\leq\log2}\eta_{\epsilon}^{v}(t)\geq\psi_{T_{k}}-x-C\right)\leq Ce^{-\left(\psi_{T_{k}}-x-C\right)^{2}/2\log2}
\]
for some absolute constant $C$. Therefore, by the three previous
displays, we obtain
\[
\mathbb{P}\left(G(\lambda)\right)\leq C\sum_{k=1}^{n}\psi_{T_{k}}\int_{-\infty}^{\psi_{T_{k}}}e^{-x\left(\sqrt{2d}-O(\log T/T)\right)}(\psi_{T_{k}}-x)e^{-\left(\psi_{T_{k}}-x-C\right)^{2}/2\log2}dx.
\]
A change of variables $u=\psi_{T_{k}}-x$ yields
\[
\mathbb{P}\left(G(\lambda)\right)\leq C\sum_{k=1}^{n}\psi_{T_{k}}e^{-\sqrt{2d}\psi_{T_{k}}}
\]
\[
=C\sum_{k=1}^{n}\left(\lambda+10\log\left(\min\left\{ T_{k},T-T_{k}\right\} \vee1\right)\right)e^{-\sqrt{2d}\left(\lambda+10\log\left(\min\left\{ T_{k},T-T_{k}\right\} \vee1\right)\right)}
\]
\[
=C\sum_{k=1}^{n}\frac{\left(\lambda+10\log\left(\min\left\{ T_{k},T-T_{k}\right\} \vee1\right)\right)}{\left(\min\left\{ T_{k},T-T_{k}\right\} \vee1\right)^{10}}e^{-\sqrt{2d}\lambda}\leq C\lambda e^{-\sqrt{2d}\lambda},
\]
where $\cdot\vee\cdot=\max\left\{ \cdot,\cdot\right\} $, and the
convergence of the last sum is due the exponent $10$ in the denominator
(with room to spare).
\end{proof}
We now finish the proof of Proposition \ref{proofrtubMBRW}. Fix $A\subset V_{\epsilon}$
and $z\in\mathbb{R}$. For $z+\left(\left|V_{\epsilon}\right|/\left|A\right|\right)^{1/4}\geq1$,
let $\lambda=z+\left(\left|V_{\epsilon}\right|/\left|A\right|\right)^{1/4}$,
and continuing with the notation of Claim 4.4, we let
\[
F_{v}=\left\{ \eta_{\epsilon}^{v}(t)\leq\frac{m_{\epsilon}}{T}t+\psi_{t}\text{ for all }0\leq t\le T,\eta_{\epsilon}^{v}\geq m_{\epsilon}+z\right\} ,
\]
where $v\in V_{\epsilon}$. We now compute
\[
\mathbb{P}\left(F_{v}(\lambda)\right)=\int_{z}^{\psi_{T}}\frac{d\mathbb{P}}{d\mathbb{Q}}(x+m_{\epsilon})\chi_{T}(x)dx
\]
\[
\leq C\int_{z}^{\psi_{T}}2^{-dn}e^{-x\left(\sqrt{2d}-O(\log T/T)\right)}\psi_{T}\left(\psi_{T}-x\right)dx
\]
\[
\leq C2^{-dn}\psi_{T}e^{-\sqrt{2d}\psi_{T}}\int_{0}^{\psi_{T}-z}e^{u}u\, du\leq C2^{-dn}\psi_{T}e^{-\sqrt{2d}z}\left(\psi_{T}-z\right).
\]
Recalling that $\psi_{T}=\lambda=z+\left(\left|V_{\epsilon}\right|/\left|A\right|\right)^{1/4}$,
we obtain
\[
\mathbb{P}\left(F_{v}(\lambda)\right)\leq C2^{-dn}\left(z+\left(\left|V_{\epsilon}\right|/\left|A\right|\right)^{1/4}\right)\left(\left|V_{\epsilon}\right|/\left|A\right|\right)^{1/4}e^{-\sqrt{2d}z}
\]
\[
\leq C2^{-dn}\left(\left|V_{\epsilon}\right|/\left|A\right|\right)^{1/2}e^{-cz}.
\]
Adding the last display for $v\in A$ and using Claim 4.4, we obtain
\[
\mathbb{P}\left(\max_{v\in A}\eta_{\epsilon}^{v}\geq m_{\epsilon}+z\right)\leq C\left(\epsilon^{d}\left|A\right|\right)^{1/2}e^{-cz}+C\left(z+\left(\left|V_{\epsilon}\right|/\left|A\right|\right)^{1/4}\right)e^{-\sqrt{2d}\left(z+\left(\left|V_{\epsilon}\right|/\left|A\right|\right)^{1/4}\right)}
\]
\[
\leq C\left(\epsilon^{d}\left|A\right|\right)^{1/2}e^{-cz}
\]
for some $0<c,C<\infty$ (depending on $d$ only), as desired. The
previous computation was made under the assumption $z+\left(\left|V_{\epsilon}\right|/\left|A\right|\right)^{1/4}\geq1$.
Assume now $\left(\left|V_{\epsilon}\right|/\left|A\right|\right)^{1/4}-1\leq-z$.
In this case,
\[
\left(\epsilon^{d}\left|A\right|\right)^{1/2}e^{-cz}\geq c\left(\epsilon^{d}\left|A\right|\right)^{1/2}e^{c\left(\epsilon^{d}\left|A\right|\right)^{-1/4}}.
\]
But $\inf_{0<x<1}x^{1/2}e^{cx^{-1/4}}\geq c>0$, where $c$ depends
only $d$. Therefore, in this case, Proposition \ref{proofrtubMBRW}
holds trivially by adjusting the constant $C$.\end{proof}


\begin{thebibliography}{10}
\bibitem[1]{k} N. V. Krylov. Introduction to the theory of random
processes. \emph{AMS, Providence, RI, 2002}.

\bibitem[2]{rf} R. J. Adler and J. E. Taylor. Random Fields and Geometry.
\emph{Springer Monographs in Mathematics. Springer, 2007.}

\bibitem[3]{bz} M. Bramson and O. Zeitouni. Tightness of the recentered
maximum of the two-dimensional discrete Gaussian free field. \emph{Comm.
Pure Appl. Math. 65:1-20, 2011}.

\bibitem[4]{d} J. Ding. Exponential and double exponential tails
for maximum of two-dimensional discrete Gaussian free field, 2011.
\emph{\url{http://arxiv.org/abs/1105.5833}}.

\bibitem[5]{bdz} M. Bramson, J. Ding and O. Zeitouni. Convergence
in law of the maximum of the two-dimensional discrete Gaussian free
field, 2013. \emph{\url{http://arxiv.org/abs/1301.6669}}.

\bibitem[6]{bdg} E. Bolthausen, J.-D. Deuschel, and G. Giacomin.
Entropic repulsion and the maximum of the two- dimensional harmonic
crystal. \emph{Ann. Probab., 29(4):1670–1692, 2001}.

\bibitem[7]{da} O. Daviaud. Extremes of the discrete two-dimensional
Gaussian free field. \emph{Ann. Probab., 34(3):962–986, 2006}.

\bibitem[8]{hmp} X. Hu , J. Miller and Y. Peres. Thick points of
the Gaussian free field. \emph{Ann. Probab., 38(2): 896–926, 2010}. 

\bibitem[9]{dy} E. B. Dynkin. Markov processes and random fields.
\emph{Bull. Amer. Math. Soc. (N.S.), 3(3):975–999, 1980}. 

\bibitem[10]{s} S. Sheffield. Gaussian free fields for mathematicians.
\emph{Probab. Theory Related Fields 139:521–541, 2007}.

\bibitem[11]{drsv} B. Duplantier, R. Rhodes, S. Sheffield and V.
Vargas. Critical Gaussian Multiplicative Chaos: Convergence of the
Derivative Martingale, 2012. \emph{\url{http://arxiv.org/abs/1206.1671}}.

\bibitem[12]{m} T. Madaule. Maximum of a log-correlated Gaussian
field, 2013. \emph{\url{http://arxiv.org/abs/1307.1365}}.

\bibitem[13]{mrv} T. Madaule, R. Rhodes and V. Vargas. Glassy phase
and freezing of log-correlated Gaussian potentials. \emph{\url{http://arxiv.org/abs/1310.5574}.}

\bibitem[14]{ce} R. J. Adler. An Introduction to Continuity, Extrema
and Related Topics for General Gaussian Processes. \emph{Lecture Notes
- Monograph Series. Institute Mathematical Statistics, Hayward, CA,
1990}. 

\bibitem[15]{dz} J. Ding and O. Zeitouni. Extreme values for two-dimensional
discrete Gaussian free field, 2012. \url{http://arxiv.org/abs/1206.0346}.

\bibitem[16]{ks} I. Karatzas and S. E. Shreve. Brownian Motion and
Stochastic Calculus. \emph{Springer-Verlag New York, NY, 1988.}\end{thebibliography}
\end{document}